\documentclass[12pt]{amsart}
\usepackage[draft=false,kerning=true]{microtype}
\usepackage[margin=1in]{geometry}
\author{Christoph Buchheim \and Maribel Montenegro \and Angelika Wiegele}
\date{}

\newtheorem{theorem}{Theorem}
\newtheorem{lemma}[theorem]{Lemma}
\newtheorem{coro}[theorem]{Corollary}

\usepackage{etex}

\usepackage{graphicx, subfigure}
\usepackage{amssymb}
\usepackage{amsmath}
\usepackage{mathtools}
\usepackage{mathrsfs}
\usepackage{rotating}
\usepackage{nicefrac}
\usepackage[linesnumbered,ruled,vlined]{algorithm2e}
\SetKwRepeat{Do}{do}{while}
\usepackage{algpseudocode}
\usepackage{pstricks}
\usepackage{pst-func,pst-eucl,pstricks-add}

\DeclareMathOperator{\rank}{rank}
\DeclareMathOperator{\cv}{conv}

\DeclareMathOperator*{\argmax}{arg\,max}

\newcommand{\Dg}[1]{\mathrm{Diag}(#1)}
\newcommand{\aop}{\mathcal{A}}
\newcommand{\Rbb}{\mathbb{R}}
\newcommand{\Rbbn}{\mathbb{R}^n}

\newcommand{\Zbb}{\mathbb{Z}}

\newcommand{\1}{{{\mathchoice {\rm 1\mskip-4mu l} {\rm 1\mskip-4mu l}{\rm 1\mskip-4.5mu l} {\rm 1\mskip-5mu l}}}}
\newcommand{\pin}[2]{\left\langle#1,#2\right\rangle}

\newcommand{\st}{\text{s.t.\ }}

\newcommand{\funjc}[3]{#1\colon #2 \longrightarrow #3}
\newcommand{\funcionjc}[5]{{\setlength{\arraycolsep}{2pt}
    \begin{array}{lccl}
      #1\colon & #2 & \longrightarrow & #3\\
      & #4 & \longmapsto & #5
    \end{array}}
}

\newcommand{\cal}{\mathcal}
\newif\ifappendix
\global\appendixtrue

\usepackage{color}
\newcommand{\mred}[1]{#1}

\begin{document}

\title{SDP-based Branch-and-Bound for Non-convex Quadratic Integer Optimization}
\thanks{This work was partially supported by the Marie Curie Initial Training Network MINO (Mixed-Integer Nonlinear Optimization) funded by the European Union. The first and the second author were partially supported by the DFG under grant BU 2313/4-2. This paper is based on the PhD thesis~\cite{Montenegro(2017)}; a
preliminary version can be found in~\cite{iscopaper}.}




\maketitle

\begin{abstract}
  Semidefinite programming (SDP) relaxations have been intensively
  used for solving discrete quadratic optimization problems, in
  particular in the binary case. For the general non-convex integer
  case with box constraints, the branch-and-bound algorithm Q-MIST has
  been proposed~\cite{BuchheimWiegele(2013)}, which is based on an extension of the well-known
  SDP-relaxation for max-cut. For solving the resulting SDPs, Q-MIST uses an off-the-shelf
  interior point algorithm.

  In this
  paper, we present a tailored coordinate ascent algorithm for solving
  the dual problems of these SDPs. Building on related ideas of
  Dong~\cite{HongboDong(2014)}, it exploits the particular structure
  of the SDPs, most importantly a small rank of the constraint
  matrices. The latter allows both an exact line search and a fast
  incremental update of the inverse matrices involved, so that the
  entire algorithm can be implemented to run in quadratic time per
  iteration. Moreover, we describe how to extend this approach to a
  certain two-dimensional coordinate update. Finally, we explain how
  to include arbitrary linear constraints into this framework, and
  evaluate our algorithm experimentally.

\keywords{Quadratic integer programming \and semidefinite programming \and
  coordinate-wise optimization}
\end{abstract}

\section{Introduction}
\label{intro}

We address integer quadratic optimization problem of the following form
\begin{align}
\label{P:generalqp}
\min \ & \ x^\top\hat Q x +\hat l^\top x +\hat c  \nonumber \\
\st & \  Ax \le b \tag{IQP}\\
    & \ x\in \Zbb^{n}, \nonumber
\end{align}
where $\hat Q$ is a symmetric $n\times n$ matrix, $\hat l\in\Rbbn$,
$\hat c\in\Rbb$, $A\in\Rbb^{m\times n}$, and $b\in\Rbb^m$.

Even in the special case of a convex objective function, i.e.,
when~$\hat Q$ is positive semidefinite, Problem~\eqref{P:generalqp} is
NP-hard in general due to the presence of integrality constraints. In
fact, in the unconstrained case it is equivalent to the NP-hard
closest vector problem~\cite{Boas(1981)}. However, dual bounds can be
computed by relaxing integrality and then solving the resulting convex
QP-relaxations. These bounds can be used within a branch-and-bound
algorithm~\cite{BuchheimDeSantis} and improved in various
ways exploiting integrality~\cite{BuchheimCaprara,BuchheimHuebner}. Dual bounds can also be
derived from semidefinite relaxations~\cite{ParkBoyd}. More generally,
convex discrete optimization problems can be addressed by solving
convex non-linear relaxations or by other approaches such as outer
approximation~\cite{bonmin}.
In the case of a non-convex objective, the problem remains NP-hard
even if integrality constraints are dropped. If only box constraints
are considered, the resulting problem is called Box-QP, it has
attracted a lot of attention in the
literature~\cite{Burer2009,BurerLetchford2009,bonami16}. \mred{Exploiting the integrality instead, the problem can be convexified using the QCR-method~\cite{BillionnetElloumiLambert(2012)}.}

For integer variables subject to box constraints and a general
quadratic objective function, a branch-and-bound algorithm called
Q-MIST has been presented by Buchheim and
Wiegele~\cite{BuchheimWiegele(2013)}. It is based on SDP formulations
that generalize the well-known \mred{semidefinite} relaxation for
max-cut~\cite{Poljak1995}.  At each node of the branch-and-bound tree,
Q-MIST calls a standard interior point method to solve a semidefinite
relaxation obtained from Problem~\eqref{P:generalqp}.  It is
well-known that interior point algorithms are theoretically efficient
to solve semidefinite programs, they are able to solve medium to small
size problems with high accuracy, but they are memory and time
consuming, becoming less useful for large-scale instances. For a
survey on interior point methods for SDP; see, e.g.,~\cite{Wolkowicz}
and \cite{Anjos2012}.

Several researchers have proposed other approaches for solving SDPs that all
attempt to overcome the practical difficulties of interior point methods.
The most common ones include bundle methods~\cite{HelmbergRendl2000}
and (low rank) reformulations as unconstrained non-convex
optimization problems together with the use of non-linear methods to solve
the resulting problems~\cite{HomerPeinado(1997),BurerMonteiro(2001),Grippo2011}.
\mred{Furthermore, algorithms based on augmented Lagrangian methods have been applied successfully for solving semidefinite programs~\cite{BurerVandenbussche(2006),MalickPovhRendlWiegele(2009),WenGoldfarbYin(2010),ZhaoSunToh(2010),SunTohYang(2015),KimKojimaToh(2016)}.}
Recently, another algorithm  has been proposed by Dong~\cite{HongboDong(2014)}
for solving a class of semidefinite programs. The author of~\cite{HongboDong(2014)} also considers
Problem~\eqref{P:generalqp} with box-constraints and reformulates it as a convex quadratically constrained problem,
then convex relaxations are produced via a cutting surface procedure based
on diagonal perturbations. The separation problem turns out
to be a semidefinite \mred{program} with convex non-smooth objective function,
and it is solved by a primal barrier coordinate minimization algorithm
with exact line search.


\paragraph{Our Contribution.}

In this paper, we focus on improving Q-MIST by using an alternative
method for solving the \mred{semidefinite} relaxation.  Our approach tries to exploit
the specific problem structure, namely a small total number of
(active) constraints and low rank constraint matrices that appear in
the semidefinite relaxation. We exploit this special structure by
solving the dual problem of the semidefinite relaxation by means of a
coordinate ascent algorithm that adapts and generalizes the algorithm
proposed in~\cite{HongboDong(2014)}, based on a barrier
model. While the main idea of exploiting the sparsity of the
  constraint matrices is taken from~\cite{HongboDong(2014)}, the class of semidefinite relaxations we obtain is much
  more general than the ones considered in~\cite{HongboDong(2014)}. In
  particular, the choice of the coordinate and the computation of optimal step lengths \mred{become} more
  sophisticated. However,
we can efficiently find a coordinate with largest gradient
entry, even if the number of constraints is exponentially large, and
perform an exact line search using the Woodbury formula.
Moreover, we can extend this idea and optimize over certain
combinations of two coordinates simultaneously, which leads to a
significant improvement of running times.

The basic idea of the approach has already been presented
  in~\cite{iscopaper}. However, a thorough mathematical analysis has
  not been given there. In particular, we show here that strong
  duality holds for the semidefinite relaxations and that the level sets of the
  barrier problem are closed and bounded, so that a coordinate ascent
  method is guaranteed to converge; this type of analysis is also missing
  in~\cite{HongboDong(2014)}. Based on this, we can now give rigorous
  proofs for the existence of optimal step lengths. Moreover, we
  introduce a more flexible SDP formulation depending on a
  vector~$\beta$ which does not change the primal feasible set, but
  the dual one, and which turns out to improve the convergence
  properties in practice when chosen appropriately.

Different from~\cite{iscopaper}, we now also explain how to
extend this method in order to include arbitrary linear constraints
instead of only box constraints. This allows to address a much
 larger class of problem instances \mred{than}~\cite{iscopaper}.
However, the main difference to~\cite{iscopaper} from a computational
  point of view is the embedding of our method into a branch-and-bound
  scheme, including a discussion of how to compute primal solutions
  from the dual solutions in order to obtain a primal heuristic. We
  investigate the branch-and-bound algorithm experimentally and show that this
method not only improves Q-MIST with respect to using a general
interior point algorithm, but also outperforms standard optimization
software for most types of instances. The experiments presented
in~\cite{iscopaper} and~\cite{HongboDong(2014)} only evaluate the
dual bounds obtained from the method, but not the total running time needed
to solve the integer problems to optimality.

\paragraph{Outline.}

This paper is organized as follows.
In  Section~\ref{Sec:preliminaries} we recall
the  semidefinite relaxation of Problem~\eqref{P:generalqp} having
box-constraints only, rewrite it
in a matrix form, compute its dual and point out the properties
of this problem that will be used later.
In Section~\ref{Sec:generalalgo} we adapt and extend the
coordinate descent algorithm presented in~\cite{HongboDong(2014)}.
Then, we improve this first approach by exploiting the special structure of the
constraint matrices. We will see that this approach can be easily adapted to
more general quadratic problems that include linear constraints, which is
presented in Section~\ref{Sec:lc}.
Finally, in Section~\ref{Sec:experiments}  we evaluate this approach within the
branch-and bound framework of Q-MIST. The experiments show that
our approach produces lower bounds of the same quality but in significantly
shorter computation time for instances of large size.

\section{Preliminaries}
\label{Sec:preliminaries}


We first consider non-convex quadratic mixed-integer optimization problems
of the form
\begin{align}
\label{P:Quadratic}
\min &\quad  x^\top \hat Q x +\hat l^\top x+ \hat c\nonumber\\
\st &\quad x\in D_1\times\dots\times D_n,
\end{align}
where~$\hat Q\in S_{n}$ is not necessarily positive
semidefinite,~$\hat l\in \Rbbn$,~$\hat c\in \Rbb$, and the feasible
domain for variable~$x_i$ is a set~$D_i=\{l_i,l_i+1,l_i+2,\dots,u_i-1,u_i\}$
for~$l_i,u_i\in\Zbb$; by~$S_{n}$ we denote the set of all
symmetric~$n\times n$-matrices.  In~\cite{BuchheimWiegele(2013)}, a
more general class of problems has been considered, allowing arbitrary
closed subsets~$D_i\subseteq\Rbb$. However, in many applications, the
set~$D_i$ is finite, and for simplicity we may
assume~$D_i=\{l_i,l_i+1,l_i+2,\dots,u_i-1,u_i\}$ then. Moreover, the algorithm presented
in the following is easily adapted to a mixed-integer setting. In
Section~\ref{Sec:lc}, we will additionally allow arbitrary linear
contraints.


\subsection{Semidefinite relaxation}
\label{Sec:sdprelaxation}

In~\cite{BuchheimWiegele(2013)} it has been proved that
Problem~\eqref{P:Quadratic} is equivalent to
\begin{align}
\label{P:SDP}
\min \quad \pin{Q}{X}\ & \nonumber\\
\st \ \ (x_{0i},x_{ii})&\in P(D_i)
\quad \forall i=1,\dots n\\
~x_{00}&=1 \nonumber\\
\rank(X) &=1 \nonumber\\
~X&\succeq0\;, \nonumber
\end{align}
where \mred{$x_{ij}$ is the element in row~$i$ and column~$j$ of matrix $X$, which is indexed by $\{0,1,\dots,n\}$,} $P(D_i):=\cv\{(u,u^2)\mid u\in D_i\}$ and the matrix~$Q\in S_{n+1}$ is defined as
$$
Q
=
\begin{pmatrix}
\hat c & \tfrac{1}{2}\hat l^\top\\
\tfrac{1}{2}\hat l & \hat Q
\end{pmatrix}
\;.$$
As only  the rank-constraint is
non-convex in this formulation, by dropping it we obtain a semidefinite relaxation of~\eqref{P:Quadratic}.

By our assumption,
the set~$D_i$ is a finite sub-set of~$\Zbb$. In this
case,~$P(D_{i})$ is a polytope in~$\Rbb^2$ with~$|D_i|$ many extreme
points. It has therefore a representation as the set of solutions of
a system of~$|D_i|$ linear inequalities. Figure~\ref{Fig:PofD} shows
two examples.

\begin{figure}
\begin{center}
\subfigure[ ][$D_i=\{0,1\}$]{
\includegraphics{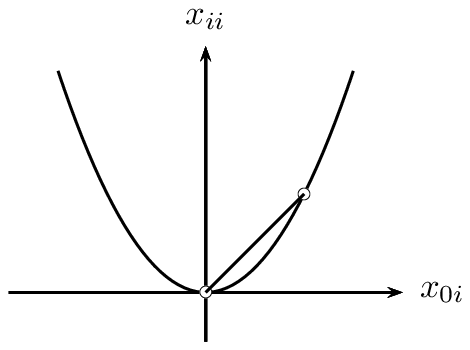}
}%
\hspace{1cm}
\subfigure[][$D_i=\{-1,0,1\}$]{
\includegraphics{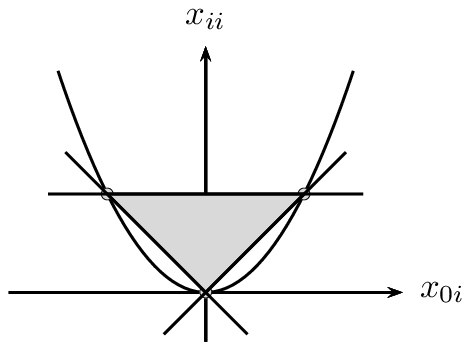}
}
\end{center}
\vspace{-0.5cm}
\caption{\label{Fig:PofD} The set~$P(D_i)$ and its polyhedral description}
\end{figure}

\begin{lemma}
\label{Lem:polyhedralformofP(Di)}
Let~$D_{i}=\{l_{i},\dots, u_{i}\}$ with~$l_{i}, u_{i}\in \Zbb$ and
$n_i:=|D_{i}|=u_i-l_i+1$. Then~$P(D_i)$ is completely described
by~$n_i-1$ lower bounding facets
\[
-x_{ii}+(2j+1)x_{0i}\leq j(j+1), \quad j=l_{i},l_{i}+1,\dots,u_{i}-1,
\]
and one upper bounding facet
\[
x_{ii}-(l_{i}+u_{i})x_{0i}\leq -l_iu_i.
\]
\end{lemma}
Notice that in case $|D_i|=2$, i.e., when the variable is binary, there is
only one lower bounding facet that together with the upper bounding
facet results in a single equation, namely, $x_{ii}-(2l_i+1)x_{0i}=-l_i(l_i+1)$.
However, for sake of simplicity, we will not distinguish these cases in the following.

\subsection{Matrix formulation}
\label{Sec:matrixform}


The relaxation of~\eqref{P:SDP} contains the
constraint~$x_{00}=1$, and this fact is exploited to rewrite the polyhedral
description of~$P(D_i)$ presented in Lemma~\ref{Lem:polyhedralformofP(Di)} as
\begin{align*}
  (\beta_{ij}-j(j+1))x_{00}-x_{ii}+(2j+1)x_{0i} & \leq \beta_{ij}, \quad j=l_{i},l_{i}+1,\dots,u_{i}-1\\
  (\beta_{iu_{i}}+l_iu_i)x_{00}+x_{ii}-(l_{i}+u_{i})x_{0i} & \leq \beta_{iu_{i}}
\end{align*}
for an arbitrary vector~$\beta\in\Rbb^m$, with~$m=\sum_{i=1}^nn_i$.  The
introduction of~$\beta$ does not change the primal
problem, but it has a strong impact on the dual
problem: the dual feasible set and objective function are both
affected by~$\beta$, as shown below.
The resulting inequalities are written in matrix form as
\[
\pin{A_{ij}}{X}\leq \beta_{ij}\;,
\]
where, for each variable~$i\in
\{1,\dots,n\}$, the index~$ij$ represents the inequalities
corresponding to lower bounding facets~$j=l_{i},l_{i}+1,\dots,u_{i}-1$
and~$j= u_{i}$ \mred{corresponding} to the upper bounding facet; see
Figure~\ref{Fig:inequality} for an illustration.

\begin{figure}
\begin{center}
\includegraphics{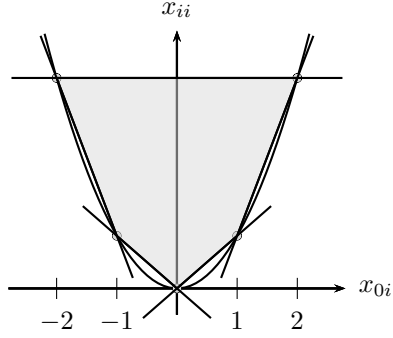}
\end{center}
\caption{\label{Fig:inequality}\small The polytope
 $P(\{-2,-1,0,1,2\})$. Lower bounding facets are indexed, from left to right,
  by~$j= -2,-1,0,1$, the upper bounding facet by 2.}
\end{figure}

Since each constraint links only the variables
$x_{00}$,~$x_{0i}$ and~$x_{ii}$, the constraint matrices~$A_{ij}\in S_{n+1}$ are sparse, the only non-zero entries being
$$
(A_{ij})_{00} = \beta_{iu_i}+l_iu_i,\quad
(A_{ij})_{0i} =(A_{ij})_{i0}=-\tfrac{1}{2}(l_{i}+u_{i}),\quad
(A_{ij})_{ii} =1
$$
in the upper bound constraint and
$$
(A_{ij})_{00} = \beta_{ij}-j(j+1),\quad
(A_{ij})_{0i} =(A_{ij})_{i0}=j+\tfrac{1}{2},\quad
(A_{ij})_{ii} =-1
$$
in the case of a lower bound constraint.
To be consistent, the constraint~$x_{00}=1$ is also written in
matrix form as~$\pin{A_0}{X}=1$, where~$A_0:=e_0e_0^\top\in~S_{n+1}$ and $e_0\in\Rbb^{n+1}$ is the unit vector $(1,0,\dots,0)^\top$.
In summary, the \mred{semidefinite} relaxation of~\eqref{P:SDP} can now be written as
\begin{align}
\label{P:SDP-matrix}
\min~~~\pin{Q}{X}& \nonumber \\
\st~~\pin{A_{0}}{X}&=1 \\
\pin{A_{ij}}{X}&\leq \beta_{ij} \quad\forall j=l_i,\dots,u_i\quad\forall i=1,\dots,n \nonumber \\
X&\succeq0. \nonumber
\end{align}
The following observation is crucial for the algorithm
presented in this paper.
\begin{lemma}
  \label{lemma:rank}
  All constraint matrices~$A_{ij}$ have rank one or two. The rank
  of~$A_{ij}$ is one if and only if
  \begin{itemize}
  \item[(a)] the facet is upper bounding, i.e.,
   $j=u_i$, and~$\beta_{iu_i}=\frac 14(l_i-u_i)^2$, or
  \item[(b)] the facet is lower bounding, i.e.,
   $j<u_i$, and~$\beta_{ij}=-\frac 14$.
  \end{itemize}
\end{lemma}
This property of the constraint matrices will be exploited later when
solving the dual problem of~\eqref{P:SDP-matrix} using a coordinate-wise
approach, leading to a computationally cheap update at each iteration
and an easy computation of the exact step size.

\subsection{Dual problem}

In order to derive the dual of Problem~\eqref{P:SDP-matrix}, we first introduce the
linear operator~$\funjc{\mathcal{A}}{S_{n+1}}{\Rbb^{m+1}}$ as
\[
\mathcal{A}(X):=
\begin{pmatrix}
\pin{A_{0}}{X}\\
\pin{A_{ij}}{X}_{j\in \{l_i,\dots,u_i\}, i\in \{1, \dots, n\}}
\end{pmatrix}
.
\]
Moreover, a dual variable~$y_{0}\in \Rbb$ is associated with the
constraint~$\pin{A_{0}}{X}=1$ and a dual variable~$y_{ij}\leq 0$ with
the constraint~$\pin{A_{ij}}{X}\leq \beta_{ij}$, for all~$j$ and~$i$, and~$y\in\Rbb^{m+1}$
is defined as
\[
y:=
\begin{pmatrix}
y_0\\
(y_{ij})_{j\in \{l_i,\dots,u_i\}, i\in \{1, \dots, n\}}
\end{pmatrix}.
\]
We thus obtain the
dual semidefinite program of Problem~\eqref{P:SDP-matrix} as
\begin{align}
\label{P:Dsdp}
\max~\quad \pin{b}{y}\quad & \nonumber \\
\st \quad  Q - \mathcal{A}^\top y &\succeq 0 \\
y_{0} &\in \Rbb \nonumber\\
y_{ij}&\leq0 \quad \forall j=l_{i},\dots,u_i \quad\forall i=1,\dots, n\nonumber,
\end{align}
the vector~$b\in \Rbb^{m+1}$ being defined as~$b_0=1$
and~$b_{ij}=\beta_{ij}$.

We conclude this section by emphasizing some characteristics of any
feasible solution of Problem~\eqref{P:SDP-matrix} that motivate the use of a  coordinate-wise optimization
method to solve the dual
problem~\eqref{P:Dsdp}; see~\cite{iscopaper} for a proof.
\begin{lemma}
\label{Lem:ActiveIneq}
Let~$X^*$ be a feasible solution of
Problem~\eqref{P:SDP-matrix}. For~$i\in\{1,\dots,n\}$, consider the
active set
\[
\mathcal{A}_i=\{j\in\{l_i,\dots,u_i\}\mid
\pin{A_{ij}}{X^*}=\beta_{ij}\}
\]
corresponding to variable~$i$. Then
\begin{itemize}
\item[(i)]  for all~$i\in\{1,\dots,n\}$,~$|\mathcal{A}_i|\le 2$, and
\item[(ii)] if~$|\mathcal{A}_i|=2$, then
 ~$x_{ii}^*=(x_{0i}^*)^2$ and~$x_{0i}^*\in D_i$.
\end{itemize}
\end{lemma}


\noindent
Lemma~\ref{Lem:ActiveIneq}\,(ii) allows
to deduce integrality of certain primal variables. If for some
primal variable, two of the corresponding dual
variables are non-zero in an optimal dual solution, then this primal
variable will be integer and hence feasible for the underlying problem.

\subsection{Primal and dual strict feasibility}
\label{Sec:strictlyfeas}
We next show that both Problem~\eqref{P:SDP-matrix} and its dual,
Problem~\eqref{P:Dsdp}, are strictly feasible.
Using this we can conclude that strong duality
holds and that both problems attain their optimal solutions.

\begin{theorem}
\label{Th:primalstrictfeas}
\mred{Problem~\eqref{P:SDP-matrix} is strictly feasible.}
\end{theorem}
\begin{proof}
  Consider the functions~$l_i(x)$ and~$u_i(x)$ bounding~$x_{ii}$ in terms
  of~$x_{0i}$, given by the upper and the lower bounding facets
  described in Lemma~\ref{Lem:polyhedralformofP(Di)}:
  \begin{eqnarray*}
    l_{i}\colon [j,j+1]  & \rightarrow \Rbb, & l_i(x) := (2j+1)x -j(j+1), \quad j=l_i,\dots,u_i-1\\
    u_{i}\colon [l_i,u_i] & \rightarrow \Rbb, & u_i(x) := (l_i+u_i)x -l_iu_i.
  \end{eqnarray*}
  Now, define~$x\in \Rbb^{n+1}$ by $x_0:=1$ and~$x_i:=\tfrac 12(l_i+u_i)$ and let the matrix $X^0$ be defined as follows
  $$x^0_{ij}:=\begin{cases}\begin{array}{ll}
    x_ix_j & \text{if }i\neq j\\
    \tfrac 12(l_i(x_i)+u_i(x_i)) & \text{otherwise.}
  \end{array}\end{cases}$$
  By the Schur complement, now $X^0\succ 0$ if and only if
  $$X^0_{\{1,\dots,n\},\{1,\dots,n\}}-X^0_{\{1,\dots,n\},0}X^0_{0,\{1,\dots,n\}}\succ
  0\;,$$ where $X^0_{I,J}$ refers to the submatrix of $X^{0}$ containing
    rows and columns indexed by $I$ and $J$, respectively. 
The latter matrix is a diagonal matrix with entries~$\tfrac
12(l_i(x_i)+u_i(x_i))-x_i^2>0$\mred{, so that the semidefinite constraint $X^0\succeq 0$ is strictly satisfied. Moreover, by construction it is clear
  that~$X^0$ satisfies all affine-linear constraints of Problem~\eqref{P:SDP-matrix}.}
\qed\end{proof}

\begin{theorem}
  \label{lem:initial}
  Problem~\eqref{P:Dsdp} is strictly feasible.
\end{theorem}
\begin{proof}
If~$Q\succ0$, we have that~$y^{0}=0$ is a feasible solution of
Problem~\eqref{P:Dsdp}.
 Otherwise, define~$a\in\Rbb^n$ by~$a_i=(A_{iu_i})_{0i}$ for
$i=1,\dots,n$. Moreover, define
\begin{align*}
\tilde{y}&:=\min\{\lambda_{min}(\hat Q)-1,0\}, \\
y_0&:=\hat c-\tilde y \sum_{i=1}^n(A_{iu_i})_{00}-1-(\tfrac{1}{2}\hat l-\tilde{y}a)^{\top}(\tfrac{1}{2}\hat l-\tilde{y}a),
\end{align*}
and~$y^{0}\in\Rbb^{m+1}$ as
\[
y^{0}:=
\begin{pmatrix}
y_{0}\\
(y_{ij})_{j\in \{l_i,\dots,u_i\}, i\in \{1, \dots, n\}}
\end{pmatrix},
\quad
y_{ij}=
\begin{cases}
\tilde{y}, &  j=u_i, i=1,\dots, n\\
0, & \text{otherwise.}
\end{cases}
\]
We have~$y^{0}_{ij}\leq0$ by construction, so it remains to
show that~$Q-\mathcal{A}^\top y^{0}\succ0$. To this end, first note that
\begin{equation*}
\tilde c:=\hat c-y_{0}-\tilde y
\sum_{i=1}^n(A_{iu_i})_{00}=1+(\tfrac{1}{2}\hat l-\tilde{y}a)^{\top}(\tfrac{1}{2}\hat l-\tilde{y}a)>
0\;.
\end{equation*}
By definition,
\begin{align*}
\label{Eq:001}
Q-\mathcal{A}^\top y^{0} &=Q-y_0A_0-\tilde y\sum_{i=1}^nA_{iu_i}
\\ &= Q-y_0A_0-\tilde y
\begin{pmatrix}
\sum_{i=1}^n(A_{iu_i})_{00} & a^\top\\
a & I_n
\end{pmatrix}
\\ &=
\begin{pmatrix}
\tilde c & (\tfrac{1}{2}\hat l-\tilde{y}a)^\top\\
\tfrac{1}{2}\hat l-\tilde{y}a & \hat Q-\tilde yI_{n}
\end{pmatrix}\;.
\end{align*}
Since $\tilde c>0$, by the Schur complement the last matrix is positive definite if
\[
(\hat Q -\tilde y I_n)-\tfrac{1}{\tilde c}(\tfrac{1}{2}\hat l-\tilde{y}a)(\tfrac{1}{2}\hat l-\tilde{y}a)^{\top}\succ 0.
\]
Denoting~$B:=(\tfrac{1}{2}\hat l-\tilde{y}a)(\tfrac{1}{2}\hat l-\tilde{y}a)^{\top}$,
we have
$\lambda_{max}(B)=(\tfrac{1}{2}\hat l-\tilde{y}a)^{\top}(\tfrac{1}{2}\hat l-\tilde{y}a)\geq0$
and thus
\begin{align*}
\lambda_{min}\left((\hat Q -\tilde y I_n)-\tfrac{1}{\tilde c}B\right)
&\geq\lambda_{min}(\hat Q -\tilde y I_n)+\tfrac{1}{\tilde c}\lambda_{min}(-B)\\
&=\lambda_{min}(\hat Q) -\tilde
y-\frac{\lambda_{max}(B)}{1+\lambda_{max}(B)}>0
\end{align*}
by definition of~$\tilde y$. We have found $y^0$ such that $y^0\leq0$
and~$Q-\aop^\top y^0\succ0$, hence we know that there exists $\epsilon>0$ small
enough such that $y^0-\epsilon\1$ is strictly feasible, i.e., such
that~$y^0-\epsilon\1<0$ and~$Q-\aop^\top(y^0-\epsilon\1)\succ0$.
\qed\end{proof}

\begin{coro}
Both Problem~\eqref{P:SDP-matrix} and its dual~\eqref{P:Dsdp} admit
optimal solutions, and there is no duality gap.
\end{coro}

\section{A coordinate ascent method}
\label{Sec:generalalgo}
We now present a coordinate-wise optimization method for solving the dual
problem~\eqref{P:Dsdp}. It is motivated by Algorithm~2
proposed in~\cite{HongboDong(2014)} and exploits the specific structure
of Problem~\eqref{P:SDP-matrix},
namely a small total number of (active) constraints, see Lemma~\ref{Lem:ActiveIneq}\,(i), and low rank
constraint matrices that appear in the semidefinite relaxation. As in~\cite{HongboDong(2014)}, the first step is to
introduce a barrier term in the objective function of
Problem~\eqref{P:Dsdp} to  model the semidefinite
constraint~$Q-\mathcal{A}^\top y\succeq 0$. We obtain
\begin{align}
\label{P:Dsdp-barrier}
\max~\quad f(y;\sigma)
&:=\pin{b}{y}+\sigma \log\det(Q-\mathcal{A}^\top y)  \nonumber \\
\st \  \  Q-\mathcal{A}^\top y &\succ 0 \\
 y_{0} &\in \Rbb  \nonumber \\
y_{ij}&\leq0 \quad \forall j=l_{i},\dots,u_i \quad\forall i=1,\dots, n\nonumber
\end{align}
for~$\sigma>0$. The barrier term tends to~$-\infty$
if the smallest eigenvalue of~$Q-\aop^\top y$ tends to zero, in other words,
if~$Q-\aop^\top y$ approaches the boundary of the semidefinite cone.
Therefore, the role of the barrier term is to prevent that dual variables will
leave the set $\{y\in\Rbb^{m+1}\mid Q-\mathcal{A}^\top y \succ0\}$.
We will see later that we do not need to introduce a barrier term for the non-negativity
constraints $y_{ij}\leq0$, as they can be dealt with directly.

Observe that~$f$ is strictly concave, indeed it is a sum of a linear function
and the~$\log\det$ function, which is a strictly concave function in the interior of the
positive semidefinite cone; see e.g.,~\cite{Helmberg(2000)}.

\begin{theorem}
\label{Th:LevelSets}
For all~$\sigma>0$ and~$z\in\Rbb$, the level set
\[
\mathscr{L}_f(z)
:= \{y_0\in\Rbb, y_{ij}\leq0 \mid Q-\aop^\top y \succ 0,
f(y;\sigma) \geq z \}
\]
of Problem~\eqref{P:Dsdp-barrier} is compact.
\end{theorem}
\begin{proof}
  First note that~$\mathscr{L}_f(z)$ is closed. Indeed, for any convergent sequence in~$\mathscr{L}_f(z)$, the limit~$\bar y$ satisfies~$Q-\aop^\top \bar y \succeq 0$ and
$f(\bar y;\sigma) \geq z$. Hence~$Q-\aop^\top \bar y \succ 0$ and thus $\bar y\in \mathscr{L}_f(z)$.

For the following, define $\mathcal{N}:=\{y\in\Rbb^{m+1}\mid y_{ij}\leq0\}$. We show that for all
$y\in\mathcal{N} \setminus \{0\}$ with~$\aop^\top y=0$,
it holds that~$\pin{b}{y}\neq0$. For this, assume that there exists
$y\in\mathcal{N}\setminus \{0\}$ such that~$\aop^\top y=0$
and~$\pin{b}{y}=0$. Then we can choose
$i'\in\{1,\dots,n\}$ and~$j'\in\{l_{i'},\dots,u_{i'}\}$ such that
$y_{i'j'}<0$. Defining~$\delta_0=\tfrac {y_0}{-y_{i'j'}}$ and $\delta_{ij}=\tfrac {y_{ij}}{-y_{i'j'}}\le 0$ for $ij\neq i'j'$, we obtain
\[
A_{i'j'}= \delta_0A_0 +\sum_{ij\neq i'j'} \delta_{ij}A_{ij}
\ \text{ and } \
b_{i'j'}=\delta_0b_0+\sum_{ij\neq i'j'} \delta_{ij}b_{ij}\;.
\]
By Theorem~\ref{Th:primalstrictfeas}, we know that there exists a strictly
feasible solution~$X^0\succ0$ of Problem~\eqref{P:SDP-matrix}, for which
\[
\pin{A_0}{X^0}= b_0 \quad\text{and}\quad
\pin{A_{ij}}{X^0}< b_{ij} \ \forall  ij.
\]
Thus
\begin{eqnarray*}
  b_{i'j'} & > & \pin{A_{i'j'}}{X^0}
   =  \delta_0\pin{A_0}{X^0}+\sum_{ij\neq i'j'} \delta_{ij} \pin{A_{ij}}{X^0}\\
  & \geq & \delta_0b_0+\sum_{ij\neq i'j'} \delta_{ij}b_{ij}= b_{i'j'},
\end{eqnarray*}
but this is a contradiction.
Secondly, observe that for all~$y\in\mathcal{N}$, it holds that
\begin{align*}
\pin{Q}{X^0} -\pin{y}{b} & \geq
\pin{Q}{X^0} -\pin{y}{\aop (X^0)} =\pin{Q -\aop^\top y}{ X^0}\\
&\geq \lambda_{\max}(Q -\aop^\top y)\lambda_{\min}(X^0).
\end{align*}
The last inequality follows by Lemma 1.2.4 in~\cite{Helmberg(2000)}.
We have that~$\lambda_{\min}(X^0)>0$ since~$X^0\succ0$. Thus
\begin{equation}
\label{Eq:feassetbounded}
\lambda_{\max}(Q -\aop^\top y)
\leq \frac{1}{\lambda_{\min}(X^0)}\big(\pin{Q}{X^0} -\pin{b}{y}\big).
\end{equation}

Since the level sets~$\mathscr{L}_f(z)$ are convex and closed,
in order to prove that they are bounded, it is enough to prove that they do not
contain an unbounded ray. We will prove thus
that for all feasible solutions~$\hat  y$ of Problem~\eqref{P:Dsdp-barrier},
and all~$y\in\mathcal{N}\setminus\{0\}$~there exists~$s$ such
that~$f(\hat  y+sy;\sigma)<z$ for all~$z\in \Rbb$.

First, consider the case when~$\aop^\top y=0$, then
\[
f(\hat  y +sy;\sigma)=\pin{b}{\hat  y}+ s\pin{b}{y}+\sigma \log\det(Q)
\]
and~$\pin{b}{y}\neq0$ as argued above.
Now, take the limit of~$f(\hat  y +sy;\sigma)$ for~$s\to\infty$:
if~$\pin{b}{y}>0$, then~$f(\hat  y +sy;\sigma)\to\infty$, but this contradicts
primal feasibility.
If, instead~$\pin{b}{y}<~0$, then~$f(\hat  y +sy;\sigma)\to-\infty$.

On the other hand,  if~$\aop^\top y\neq0$,  we may have
$\lambda_{\min}(Q-\aop^\top\hat  y -s^*\aop^\top y)=0$ for some~$s^*>0$,
and hence
\[
\lim_{s\to s^*} \log\det(Q-\aop^\top\hat  y -s\aop^\top y)=-\infty\;.
\]
Otherwise, $\lambda_{\min}(Q-\aop^\top\hat  y -s\aop^\top y)> 0$ for all~$s> 0$ and hence
\[
\lim_{s\to \infty} \lambda_{\max}(Q-\aop^\top \hat  y -s\aop^\top y)=\infty,
\]
and from~\eqref{Eq:feassetbounded} it follows that~$\pin{b}{\hat  y +s y}$
must tend to~$-\infty$ when~$s\to\infty$.
In the second case,
observe that~$p(s):=\det(Q-\aop^\top \hat  y -s\aop^\top y)$ is
a polynomial in~$s$,  and denote~$h(s):=\pin{b}{\hat  y+sy}
=\pin{b}{\hat  y}+\pin{b}{y}s$.
We have that
\[
\lim_{s\to\infty}\frac{\log p(s)}{h(s)}=
\lim_{s\to\infty}\frac{\frac{p'(s)}{p(s)}}{\pin{b}{y}}=
\lim_{s\to\infty}\frac{p'(s)}{\pin{b}{y}p(s)}=0.
\]
This means that~$h(s)$ dominates~$\log p(s)$ when~$s\to\infty$.
Thus~$f(\hat  y+sy)\to-\infty$ for~$s\to\infty$.
\qed\end{proof}
The boundedness of the
upper level sets and the strict concavity of the objective function
guarantee the convergence of a coordinate ascent method, when
using the cyclical rule to select the coordinate direction and exact
line search to compute the step length~\cite{Ortega}.
However, for practical performance reasons, we apply  the
Gauss-Southwell rule to choose the coordinate direction.
Below we describe a general algorithm to solve Problem~\eqref{P:Dsdp-barrier}
in a coordinate-wise maximization manner.

\begin{algorithm}
\NoCaptionOfAlgo
\caption{\textbf{Outline of a barrier coordinate ascent algorithm for
Problem~\eqref{P:Dsdp}}}

\label{Algo:general}
\begin{algorithmic}[1]
\State \textbf{Starting point:} choose $\sigma >0$ and any feasible solution~$y$
of~\eqref{P:Dsdp}.
\State \textbf{Direction:} choose a coordinate direction~$e_{ij}$.
\State \textbf{Step size:} using exact line search, determine the
step length~$s$.
\State \textbf{Move along chosen coordinate:} ~$y \gets y+se_{ij}$.
\State \textbf{Decrease} the barrier parameter~$\sigma$.
\State \textbf{Go to (2)}, unless some stopping criterion is satisfied.
\end{algorithmic}
\end{algorithm}

\noindent
In the following sections, we will explain each step of this algorithm in
detail. We propose to choose the ascent direction based on a
coordinate-gradient scheme, similar to~\cite{HongboDong(2014)}.
We thus need to compute the  gradient of the
objective function of Problem~\eqref{P:Dsdp-barrier}.
See, e.g.,~\cite{Helmberg(2000)} for more details on how to compute
the gradient. We have that
~\[
\nabla_{y}f(y;\sigma)=b-\sigma\mathcal{A}((Q-\mathcal{A}^\top y)^{-1}).
\]
For the following, we denote
~\begin{align*}
W&:=(Q-\mathcal{A}^\top y)^{-1},
\end{align*}
so that
~\begin{equation}
\label{gradient}
\nabla_{y}f(y;\sigma)=b-\sigma\mathcal{A}(W)\;.
\end{equation}

We will see that, due to the particular structure of the gradient
of the objective function, the search of the ascent direction reduces
to considering only a few possible candidates among the
exponentially many directions.
In the chosen direction, we solve a one-dimensional minimization
problem to determine the step size.
It turns out that this problem has a closed form solution. Each
iteration of the algorithm involves the update of the vector of dual variables
and the computation of~$W$, i.e., the inverse of an~$(n+1)\times(n+1)$-matrix that only
changes by a factor of one constraint matrix when changing the value of the
dual variable. Thanks to the Woodbury formula
and to the fact that our constraint matrices are rank-two matrices,
the matrix~$W$ can be easily computed incrementally\mred{. Indeed,}
the updates at each iteration of the
algorithm can be performed in~$O(n^{2})$ time, which is crucial for the
performance of the algorithm proposed.
In fact, the special structure of Problem~\eqref{P:SDP-matrix}
can be exploited even more, considering the fact that the
constraint matrix associated with the dual variable~$y_0$ has rank-one,
and that every linear combination with another linear constraint matrix still
has rank at most two.
This suggests that we can perform a plane-search rather
than a line search, and simultaneously update two dual variables and still
recompute $W$ in~$O(n^2)$ time (see Section~\ref{Sec:SimultaneusUpdate}).
Thus, the main ingredient of our algorithm is the computationally
cheap update of~$W$ at each iteration and an easy computation of the optimal step size.

Before describing in detail how to choose an ascent direction and how to
compute the step size, we address the choice of a feasible starting point.
Compared to~\cite{HongboDong(2014)}, the situation is  more complex. We propose to choose
as starting point the vector $y^0$ defined in the proof of
Theorem~\ref{lem:initial}.
The construction described there can be directly implemented, however,
it involves the computation of the smallest eigenvalue of~$\hat Q$.

\subsection{Choice of an ascent direction}
\label{Sec:direction}

We improve the objective function coordinate-wise: at each
iteration~$k$ of the algorithm, we choose an ascent
direction~$e_{ij^{(k)}}\in\Rbb^{m+1}$ where~$ij^{(k)}$ is a coordinate
of the gradient with maximum absolute value
\begin{equation}
\label{Eqn:choosecoord}
ij^{(k)}\in\argmax_{ij}|\nabla_{y}f(y;\sigma)_{ij}|\;.
\end{equation}
However, moving a coordinate~$ij$ to a positive direction is allowed
only in case~$y_{ij}<0$, so that the coordinate~$ij^{(k)}$
in~\eqref{Eqn:choosecoord} has to be chosen among those satisfying either
$\nabla_{y}f(y;\sigma)_{ij}>0$ and $y_{ij}<0$, or
$\nabla_{y}f(y;\sigma)_{ij}<0$.
The entries of the gradient depend on the type of
inequality. By~\eqref{gradient}, we have
\begin{align*}
&\nabla_{y}f(y;\sigma)_{ij}=\beta_{ij}-\sigma \pin{W}{A_{ij}}\\
&=
\begin{cases}
\beta_{ij}-\sigma((\beta_{ij}-j(j+1))w_{00}+(2j+1)w_{0i} -w_{ii}) &  j=l_i,\dots,u_i-1,\\
\beta_{iu_i}-\sigma((\beta_{iu_i}+l_{i}u_{i})w_{00} -(l_{i}+u_{i})w_{0i} +w_{ii}) &  j=u_i.
\end{cases}
\end{align*}
The number of lower bounding facets for a single primal variable~$x_i$
is~$u_i-l_i$, which is not polynomial in the input size from a
theoretical point of view. From a practical point of view, a large
domain~$D_i$ may slow down the coordinate selection if all potential
coordinates have to be evaluated explicitly.

However, the regular structure of the gradient entries corresponding
to lower bounding facets for variable~$x_i$ allows to limit the search
to at most three candidates per variable. To this end, we define the
function
\[
\funcionjc{\varphi_i}{[l_i,u_i-1]}{\Rbb}{j}{\beta_{ij}-\sigma((\beta_{ij}-j(j+1))w_{00}+(2j+1)w_{0i} -w_{ii})\;.}
\]
Our task is then to find a minimizer of~$|\varphi_i|$
over~$\{l_i,\dots,u_i-1\}$. As $\varphi_i$ is a uni-variate quadratic function,
we can restrict our search to at most three candidates, namely the bounds~$l_i$
and~$u_i-1$ and the global minimizer 
$
\tfrac{w_{0i}}{w_{00}}-\tfrac{1}{2}
$
of~$\varphi_i$ rounded to the next integer. The latter value is only taken
into account if it belongs to~$\{l_i,\dots,u_i-1\}$.
In summary, taking into account also the upper bounding facets and the
coordinate zero, we need to test at most~$1+4n$ candidates in order to
solve~\eqref{Eqn:choosecoord}, independent of the sets~$D_i$.

\subsection{Computation of the step size}
\label{Sec:steponedimension}

We compute the step size~$s^{(k)}$ by exact line search in the chosen
direction. For this we need to solve the following one-dimensional
maximization problem
\begin{equation*}
s^{(k)}= \argmax_{s}\{f(y^{(k)} +se_{ij^{(k)}};\sigma)\mid
Q-\mathcal{A}^\top (y^{(k)}+se_{ij^{(k)}}) \succ 0 , s\leq -y_{ij^{(k)}} \}\;,
\end{equation*}
unless the chosen coordinate is zero, in which case $s$ does not
have an upper bound.
Note that the function~$s\mapsto f(y^{(k)} +se_{ij^{(k)}};\sigma)$
is strictly concave on~$\{s\in\Rbb\mid Q-\mathcal{A}^\top (y^{(k)}+se_{ij^{(k)}}) \succ
0\}$. We thus need to find
an~$s^{(k)}\in\Rbb$ satisfying the semidefinite constraint
$Q-\mathcal{A}^\top (y^{(k)}+s^{(k)}e_{ij^{(k)}}) \succ 0$ such that
either
\[
\nabla_{s}f(y^{(k)} +s^{(k)}e_{ij^{(k)}};\sigma)=0
\quad\text {and}\quad y_{{ij}^{(k)}}+s^{(k)}\leq 0
\]
or
\[
\nabla_{s}f(y^{(k)} +s^{(k)}e_{ij^{(k)}};\sigma)>0
\quad\text {and}\quad s^{(k)}=- y_{ij^{(k)}}^{(k)}.
\]
In order to simplify the notation, we omit the index~$(k)$ in the
following. From the definition, we have
\begin{align*}
f(y +se_{ij};\sigma)&= \pin{b}{y}+s\pin{b}{e_{ij}}
+\sigma\log\det(Q-\mathcal{A}^\top y -s\mathcal{A}^\top e_{ij})\\
&=\pin{b}{y}+\beta_{ij} s+\sigma\log\det(W^{-1} -sA_{ij}).
\end{align*}
Then, the gradient with respect to~$s$ is
\begin{equation}
\label{Eq:gradient}
\nabla_{s}f(y +se_{ij};\sigma)=\beta_{ij}-\sigma\pin{A_{ij}}{(W^{-1} -sA_{ij})^{-1}}.
\end{equation}
The next lemma states that if the coordinate direction is chosen as
explained in the previous section, and the gradient~\eqref{Eq:gradient} has
at least one root in the right direction of the line search, then
there exists a feasible step length.

\begin{lemma}\label{lem:existence}
~
\begin{itemize}
\item[(i)] Let the coordinate~$ij$ be chosen such
  that~$\nabla_{y}f(y;\sigma)_{ij}>0$ and~$y_{ij}<0$.
  If there exists~$s\geq0$ with $\nabla_{s}f(y +se_{ij};\sigma)=0$,
  then for the smallest~$s^+\ge 0$ with $\nabla_{s}f(y +s^+e_{ij};\sigma)=0$, one of the following holds:
  \begin{itemize}
  \item[(a)] $y+s^+e_{ij}$ is dual feasible
  \item[(b)] $s^+>-y_{ij}$, $y-y_{ij}e_{ij}$ is dual feasible, and $\nabla_{s}f(y-y_{ij}e_{ij};\sigma)>0$.
  \end{itemize}
\item[(ii)] Let the coordinate~$ij$ be chosen such
  that~$\nabla_{y}f(y;\sigma)_{ij}<0$.
  If there exists some~$s\leq0$ with~$\nabla_{s}f(y +se_{ij};\sigma)=0$,
  then for the biggest $s^-\le 0$ such that~$\nabla_{s}f(y +s^-e_{ij};\sigma)=0$
  it holds that $y+s^-e_{ij}$ is dual feasible.
\end{itemize}
\end{lemma}
\begin{proof}
For showing (i), we consider the cases $s^+\le -y_{ij}$ and
  $s^+> -y_{ij}$, implying~(a) and~(b), respectively. In the first
  case, we have~$y+s^+e_{ij}\in {\cal N}$, hence it remains to
  show~$Q-\mathcal{A}^\top(y+s^+e_{ij})\succeq 0$. Assuming the
  opposite, there would exist~$0< s'\le s^+$ with~$f(s,\sigma)\to
  -\infty$ for $s\to s'$.  From the continuous differentiability
  of~$f(s,\sigma)$ on the feasible region and since
  $\nabla_{y}f(y;\sigma)_{ij}>0$, there exists~$0\le s''\le s'$ with
  $\nabla_{s}f(y +s''e_{ij};\sigma)=0$, in contradiction to the
  minimality of~$s^+$.

Otherwise, if $s^+> -y_{ij}$, by the same reasoning, we may
  assume that $y+se_{ij}$ is dual feasible for all $s\in[0,s^+)$.
  Since there is no $s'\in [0,s^+)$ with $\nabla_{s}f(y
    +s'e_{ij};\sigma)=0$, we must have $\nabla_{s}f(y
    +s'e_{ij};\sigma)>0$ for all $s'\in [0,s^+)$, again by continuous
      differentiability and $\nabla_{y}f(y;\sigma)_{ij}>0$. Now $-y_{ij}\in[0,s^+)$ and hence
      (b) follows, which concludes the proof of~(i).

Assertion (ii) now follows analogously to the first part of (i), since we always have $s^-\le 0\le -y_{ij}$.
\qed\end{proof}

If in addition we exploit that the level sets of the function are bounded,
as shown by Theorem~\ref{Th:LevelSets}, then we can derive the following
theorem. It shows that we can always choose an appropriate step length by considering the roots of the gradient~\eqref{Eq:gradient}.
\begin{theorem}\label{lem:onedimstep}
~
\begin{itemize}
\item[(i)] Let the coordinate~$ij$ be chosen such
  that~$\nabla_{y}f(y;\sigma)_{ij}>0$ and~$y_{ij}<0$.
  If the gradient~\eqref{Eq:gradient} has at least one positive root, then for the smallest positive root~$s^+$, either~$y+s^+e_{ij}$~is dual
  feasible and~$\nabla_{s}f(y +s^+e_{ij};\sigma)=0$,
  or $y_{ij}+s^+>0$ and~$\nabla_{s}f(y -y_{ij}e_{ij};\sigma)>0$. Otherwise $\nabla_{s}f(y -y_{ij}e_{ij};\sigma)> 0$.
\item[(ii)] Let the coordinate~$ij$ be chosen
  such that~$\nabla_{y}f(y;\sigma)_{ij}<0$.
  Then the gradient~\eqref{Eq:gradient} has at least one negative root,
  and for the biggest negative root~$s^-$, we have that $y+s^-e_{ij}$~is dual
  feasible and $\nabla_{s}f(y +s^-e_{ij};\sigma)=0$.
\end{itemize}
\end{theorem}

\begin{proof}
  The first part of~(i) follows directly from Lemma~\ref{lem:existence}\,(i). If the gradient~\eqref{Eq:gradient} has no positive root, continuous differentiability of $s\mapsto f(y +se_{ij};\sigma)$ together with $\nabla_{y}f(y;\sigma)_{ij}>0$ implies  $\nabla_{s}f(y +se_{ij};\sigma)> 0$ for all~$s\ge 0$.

  To show~(ii), consider the ray $y+s e_{ij}$, $s\le 0$. This ray belongs to~$\cal N$, but the level set~$\mathscr{L}_f(z)$ is bounded for $z:=f(y;\sigma)$ by Theorem~\ref{Th:LevelSets}. We derive that
$f(y+se_{ij};\sigma)<z$ for some~$s<0$. Again using continuous differentiability, we derive that there exists
    $s'\in(s,0)$ such that $\nabla_{s}f(y +s'e_{ij};\sigma)=0$. The remaining statements then follow from Lemma~\ref{lem:existence}\,(ii).
\qed\end{proof}
\mred{Theorem~\ref{lem:onedimstep} shows that we can always find an appropriate step length for the chosen coordinate~$ij$. If, according to the gradient~$\nabla_{y}f(y;\sigma)_{ij}$, we desire to increase variable~$ij$, then part~(i) shows that in all possible cases we can find a feasible step length: either it is the first root of the gradient or -- if this root is positive and hence infeasible, or if it does not exist -- we can stop when the variable turns zero. This case distinction is illustrated in Figure~\ref{Fig:optstep}, where we draw the gradient (vertical axis) in terms of the steplength (horizontal axis) and the point where variable~$ij$ turns zero is marked by a dashed line. When decreasing the variable~$ij$ as in part~(ii), the situation is simpler, as there exists no lower bound on the variables.}
\begin{figure}
\begin{center}
\subfigure[ ][$s^+\leq s$]{
\includegraphics{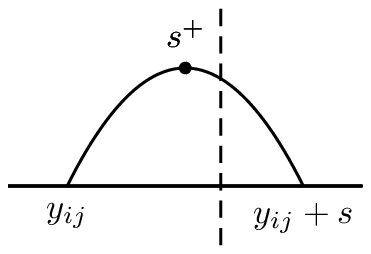}
}%
\subfigure[ ][$s^+ = -y_{ij}$]{
\includegraphics{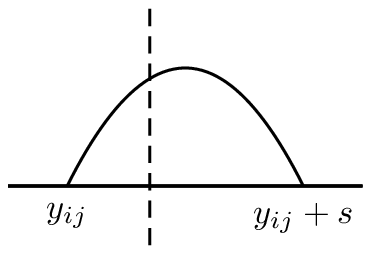}
}
\end{center}
\vspace{-0.5cm}
\caption{\label{Fig:optstep} Illustration of the existence of an optimal step size $s^+$, Theorem~\ref{lem:onedimstep}~(i)}
\end{figure}

Observe that the computation of the gradient requires to compute the inverse
of~$W^{-1} -sA_{ij}$,  it is worth mentioning that this is the crucial task since
it is a matrix of order~$n+1$. Notice\mred{,} however\mred{,} that~$W^{-1}$ is changed
by a rank-one or rank-two matrix~$sA_{ij}$; see Lemma~\ref{lemma:rank}.
Therefore, we will compute the inverse matrix~$(W^{-1} -sA_{ij})^{-1}$ using the
Woodbury formula for the rank-one or rank-two update.
\ifappendix
The computation is detailed in Appendix~\ref{Sec:stepsizeCD}.
\else
\mred{The detailed computation can be found in Chapter~5 of \cite{Montenegro(2017)}.}
\fi

\subsection{Algorithm overview and running time}
\label{Sec:overview}
Our approach to solve Problem~\eqref{P:Dsdp} is summarized in
Algorithm~\ref{algo:cd}.
\begin{algorithm}
\SetAlgoRefName{CD}
\DontPrintSemicolon
\KwIn{$Q\in S_{n+1}$}
\KwOut{A lower bound on the optimal value of Problem~\eqref{P:SDP-matrix}}
Use Theorem~\ref{lem:initial} to compute~$y^{(0)}$ such that~$Q-\mathcal{A}^\top y^{(0)} \succ 0$\;
Compute~$W^{(0)} \gets (Q-\mathcal{A}^\top y^{(0)})^{-1}$\;
\For{$k =0,1,2,\dots$} {
Choose a coordinate direction~$e_{ij^{(k)}}$ as described in Section~\ref{Sec:direction}\label{cd:algo:step4}\;
Compute the step size~$s^{(k)}$ as described in Section~\ref{Sec:steponedimension}\label{cd:algo:step5}\;
Update~$y^{(k+1)}\gets y^{(k)}+s^{(k)}e_{ij^{(k)}}$\label{cd:algo:step6}\;
Update~$W^{(k)}$ using the Woodbury formula \;
Update~$\sigma$\label{cd:algo:step8}\;
Terminate if some stopping criterion is met\label{cd:algo:step9}\;
}
\Return{$\pin{b}{y^{(k)}}$}\;
\caption{Barrier coordinate ascent algorithm for Problem~\eqref{P:Dsdp}}
\label{algo:cd}
\end{algorithm}
As already discussed in~\cite{iscopaper}, Algorithm~\ref{algo:cd} can be implemented such that its running time is~$O(n^3)$ for the preprocessing (Steps~1--2) and~$O(n^2)$ for each iteration (Steps~4--9), using the Woodbury formula and considering that  only~$O(n)$ candidates for
the coordinate selection have to be checked. Note that the vector~$y^{(k)}$ is dual feasible and hence yields a valid lower bound~$\pin{b}{y^{(k)}}$ at every iteration. Within a branch-and-bound framework, we may thus stop Algorithm~\ref{algo:cd} as soon as the current best upper bound is reached.

Otherwise, the algorithm can be stopped after a fixed number of
iterations or when other criteria show that only a small further
improvement of the bound can be expected.
The choice of an appropriate termination rule however is closely
related to the update of~$\sigma$ performed in
Step~\ref{cd:algo:step8}\mred{. This} is further discussed in Section~\ref{Sec:experiments}.

\subsection{Two dimensional approach}
\label{Sec:SimultaneusUpdate}

Algorithm~\ref{algo:cd} is based on the fact that all constraint matrices in~\eqref{P:SDP-matrix}  have rank
at most two, so that the matrix~$W^{(k)}$ can be updated  in~$O(n^2)$
time using the Woodbury formula. Considering  the special structure of the first constraint matrix~$A_0$, it is easy to verify that the rank of any linear combination of any constraint matrix~$A_{ij}$ with~$A_0$ still has rank at most two. In the following, we thus describe an extension of Algorithm~\ref{algo:cd} using a simultaneous update of both corresponding dual coordinates.
Geometrically, we thus search along the plane spanned by the
coordinates~$(e_0,e_{{ij}^{(k)}})$ rather than the line spanned by a
single coordinate~$e_{{ij}^{(k)}}$. For sake of readability,
we again omit the index~$(k)$ in the following.


Let~$ij$ be a given coordinate and denote by~$s$ the step size along
coordinate~$e_{ij}$ and by~$s_0$ the step size along~$e_0$. At each
iteration we then perform an update of the form $y\gets
y+s_0e_0+se_{ij}$.  The value of the objective function in the new
point is
\[
f(y +s_0e_0+se_{ij};\sigma)
= \pin{b}{y}+s_0+s\beta_{ij}+\sigma\log\det(W^{-1} -s_0A_0-sA_{ij})\;.
\]
To obtain a closed formula for the optimal step
length~$s_0$ in terms of a fixed step length~$s$, we exploit
the fact that the update of coordinate~$e_0$ is rank-one, and that the zero
coordinate does not have a sign restriction. Consider the gradient of
$f(y +s_0e_0+se_{ij};\sigma)$ with respect to~$s_0$:
\begin{equation}
\label{Eq:Grad_s0}
\nabla_{s_0}f(y +s_0e_0+se_{ij};\sigma)
= 1-\sigma\pin{A_0}{(W^{-1} -s_0A_0-sA_{ij})^{-1}}\;.
\end{equation}
Defining~$W(s):=(W^{-1} -sA_{ij})^{-1}$ and using the Woodbury formula for
rank-one update, we obtain
\begin{align*}
(W^{-1} -s_0A_0-sA_{ij})^{-1}  &= (W(s)^{-1} - s_0A_0)^{-1} \\
&= W(s)+\frac{s_0}{1-s_0w(s)_{00}} (W(s)e_{0}) (W(s)e_{0})^\top.
\end{align*}
Substituting the last expression in the gradient \eqref{Eq:Grad_s0} and
setting the latter to zero, we get
\[
s_0(s):= s_0 =\frac{1}{{w(s)}_{00}}-\sigma.
\]
It remains to compute~${w(s)}_{00}$, which can be done using the Woodbury
formula for rank-two updates. 
\ifappendix
See Appendix~\ref{Sec:stepsizeCD2D} for an explicit expression.
\else
\mred{See Chapter~5 of \cite{Montenegro(2017)} for an explicit expression.}
\fi
In summary, we have shown
\begin{lemma}
\label{lemma:steps02D}
Let~$s$ be a given step size along coordinate direction~$e_{ij}$, then
\begin{equation}
\label{Eq:steps0}
s_0 =\frac{1}{{w(s)}_{00}}-\sigma
\end{equation}
is the unique maximizer of~$f(y +s_0e_0+se_{ij};\sigma)$, and hence the
optimum step size along coordinate~$e_0$.
\end{lemma}

The next task is to compute a step length~$s$ such that~$(s_0(s),s)$
is an optimal two-dimensional step in the coordinate plane spanned by
$(e_0,e_{{ij}})$. To this end, we consider the function
\[
g_{ij}(s) := f(y +s_0(s)e_0+se_{ij};\sigma)
\]
over the set~$\{s\in\Rbb\mid Q-\mathcal{A}^\top (y+s_0(s)e_0+se_{ij}) %
\succ0\}$ and solve
the problem
\begin{equation}
\label{P:simultLinesearch}
\max_{s}\;\{g_{ij}(s)\mid%
Q-\mathcal{A}^\top (y^{(k)}+s_0(s)e_0+se_{ij^{(k)}}) \succ 0 , s\leq -y_{ij}^{(k)} \}\;.
\end{equation}
Since the latter problem is uni-variate and differentiable, we need to find~$s\in\Rbb$ such that
either $g_{ij}'(s)=0$ and $s\leq -y_{ij}$ or $g_{ij}'(s) > 0$ and $s= -y_{ij}$.
The derivative of~$g_{ij}(s)$ is
\begin{equation}
\label{Eq:grad2d}
g_{ij}'(s)
=s_0'(s)+\beta_{ij} -\sigma \pin{s_0'(s)A_0+A_{ij}}{(W^{-1} -s_0(s)A_0-sA_{ij})^{-1}},
\end{equation}
which is a quadratic rational function.
The next lemma shows that at least one of the two roots of~$g_{ij}'(s)$
leads to a feasible update if the direction~$ij$ is an ascent direction.
Similar to Theorem~\ref{lem:onedimstep} in the one dimensional approach,
the proof is based on Theorem~\ref{Th:LevelSets}.

\begin{theorem}~\label{lem:steptwodim}
\begin{itemize}
\item[(i)] Let the coordinate~$ij$ be chosen such that~$g_{ij}'(0)>0$ and~$y_{ij}<0$.
  If~\eqref{Eq:grad2d} has at least one positive solution, then
  for the smallest such solution~$s^+$, either the point
  $y+s_0(s^+)e_0+s^+e_{ij}$ is dual feasible and~$g_{ij}'(s^+)=0$,
  or $y_{ij}+s^+>0$ and~$g_{ij}'(-y_{ij})>0$. Otherwise $g_{ij}'(-y_{ij})>0$.

\item[(ii)] Let the coordinate~$ij$ be such that~$g_{ij}'(0)<0$.
  The expression~\eqref{Eq:grad2d} has at least one negative solution,
  and for the biggest such solution~$s^-$, the point~$y+s_0(s^-)e_0+s^-e_{ij}$~is
  dual feasible and~$g_{ij}'(s^-)=0$.
\end{itemize}
\end{theorem}

It remains to discuss the choice of the coordinate~$ij$,
which is similar to the one-dimensional approach:
we choose the coordinate direction~$e_{{ij}}$ such that
\begin{equation}
\label{P:coordij2d}
ij\in\argmax_{ij}|g_{ij}'(0)|\;,
\end{equation}
where moving into the positive direction of a coordinate~$e_{ij}$
is allowed only if~$y_{ij}<0$, thus the candidates are those coordinates
satisfying
\[
(g_{ij}'(0)>0 \text{ and } y_{ij}<0)
\quad \text{ or }
\quad
g_{ij}'(0)<0.
\]
We have that
\[
g_{ij}'(0)=
\begin{cases}
j(j+1)-2\frac{w_{0i}}{w_{00}}j-\frac{w_{0i}}{w_{00}}-(\sigma w_{00}-1)\frac{w_{ii}^2}{w_{00}^2}+\sigma w_{ii} &  j=l_i,\dots,u_i-1,\\
l_{i}u_{i}+ \frac{w_{0i}}{w_{00}}(l_{i}+u_{i}) +(\sigma w_{00}-1)\frac{w_{ii}^2}{w_{00}^2}-\sigma w_{ii} &  j=u_i,
\end{cases}
\]
\ifappendix
see Appendix~\ref{Sec:stepsizeCD2D} again. 
\else
\mred{see Chapter~5 of \cite{Montenegro(2017)} again.}
\fi
Therefore, as before, we do not need to search over all potential
coordinates~$ij$, since the regular structure of~$g_{ij}'(0)$ for the
lower bounding facets again allows us to restrict the search to at
most three candidates per variable. Thus only~$4n$ potential
coordinate directions need to be considered.

Using these ideas, a slightly different version of
Algorithm~\ref{algo:cd} is obtained by changing
Steps~\ref{cd:algo:step4},~\ref{cd:algo:step5} and~\ref{cd:algo:step6}
adequately,
we call it Algorithm~CD2D.
In Section~\ref{Sec:experiments}, we compare Algorithm~\ref{algo:cd} and its
improved version, Algorithm~CD2D, experimentally.

\subsection{Primal solutions}
\label{Sec:primal}

This section contains an algorithm to compute an approximate solution
of Problem~\eqref{P:SDP-matrix} using the information given by the
dual optimal solution of Problem~\eqref{P:Dsdp}.  We will prove that
under some additional conditions the approximate primal solution
produced is actually the optimal solution, provided that an optimal
solution~$y^*$ for the dual problem~\eqref{P:Dsdp} is given.  First
note that the primal optimal solution~$X^*\in S_{n+1}^+$ must satisfy the
complementarity condition
\begin{equation}
\label{eq:complementaryslakness}
(Q-\mathcal{A}^\top y^*)X^* =0
\end{equation}
and the primal feasibility conditions~$X^*\succeq0$ and
\begin{equation}
\label{eq:activeconstraints}
\begin{cases}
\pin{A_0}{X^*} &= 1, \\
\pin{A_{ij}}{X^*} &= \beta_{ij} \quad \forall i,j\in \mathcal{A}(y^*),
\end{cases}
\end{equation}
where~$\mathcal{A}(y^*):=\{i,j\mid y_{ij}<0\}$.

Notice that in order to find a primal optimal solution~$X^*$, we need to
solve a semidefinite program, and this is in general computationally
too expensive.
Since this has to be done at every node of the branch-and-bound tree,
we need to devise an alternative method to compute an approximate matrix~$X$
that will be used mainly for taking a branching decision
in Algorithm Q-MIST.
The idea is to ignore the semidefinite constraint $X\succeq0$.
We thus proceed as follows. We consider the spectral
decomposition~$Q- \mathcal{A}^\top y^*=P\Dg{\lambda}P^\top$.
Since~$Q- \mathcal{A}^\top y^*\succeq0$, we have~$\lambda \geq 0$.
Define~$Z:=P^\top XP$, then~$X= PZP^\top$ and~\eqref{eq:complementaryslakness} is
equivalent to
\[
0=( P\Dg{\lambda}P^\top)(PZP^\top) = P\Dg{\lambda}ZP^\top.
\]
Since~$P$ is a regular matrix, the last equation implies that~$\Dg{\lambda}Z=0$,
which is at the same time equivalent to say that~$z_{ij}=0$
whenever~$\lambda_i>0$ or~$\lambda_j>0$.
Replacing also~$X=PZP^\top$ in~\eqref{eq:activeconstraints}, we have
\begin{align*}
1 &=\pin{A_0}{X} =\pin{A_{0}}{PZP^\top} = \pin{P^\top A_{0}P}{Z},\\
\beta_{ij} &=\pin{A_{ij}}{X} = \pin{A_{ij}}{PZP^\top}= \pin{P^\top A_{ij}P}{Z}.
\end{align*}
This suggests, instead of solving the system~\eqref{eq:complementaryslakness}
and~\eqref{eq:activeconstraints} in order to compute~$X$,
\mred{solving} the system above and then \mred{computing}~$X=PZP^\top$.
The system above can be simplified, since~$Z$ has a zero row/column for
each~$\lambda_l>0$.
Thus it is possible to reduce the dimension of the problem as follows:
let~$\bar A$ be the sub-matrix of~$A$ where all rows and columns~$l$
with~$\lambda_l > 0$ are removed; let~$r$
be the number of positive entries of~$\lambda$. \mred{Letting}~$Y\in S_{n+1-r}$,
we have that the system above is equivalent to
\begin{equation}
\label{eq:reducedsystem}
\begin{cases}
\pin{\overline{P^\top A_{0}P}}{Y} &= 1 \\
\pin{\overline{P^\top A_{ij}P}}{Y} &=\beta_{ij} \quad \forall i,j\in \mathcal{A}(y^*).
\end{cases}
\end{equation}
Then we can extend~$Y$ by zeros to obtain a matrix~$Z \in S_{n+1}$, and finally compute
$X = PZP^\top$. We formulate this procedure in Algorithm~\ref{algo:primalsol}.

\mred{In practice, since we use a barrier approach to solve the semidefinite program~\eqref{P:SDP-matrix}, no entry of~$\lambda$ will be exactly zero. However, it is easy to see that in theory at least one entry of~$\lambda$ must be zero in an optimal solution to~\eqref{P:SDP-matrix}.} In the implementation of the algorithm, we \mred{thus} consider the
smallest eigenvalue of~$Q-\mathcal{A}^\top y$ as zero, this means that $r$ is
at least one, and there may be more eigenvalues considered as zero,
depending on the allowed tolerance.

\begin{algorithm}
\DontPrintSemicolon
\KwIn{optimal solution $y^* \in \Rbb^{m+1}$ of Problem~\eqref{P:Dsdp}}
\KwOut{$X\in S_{n+1}$}
Compute~$P\in \Rbb^{(n+1)\times(n+1)}$ orthogonal and~$\lambda\geq0$ with~$Q-\mathcal{A}^\top y^* = P\Dg{\lambda}P^\top$\;
Find a solution~$Y\in S_{n+1-r}$ of the system of
equations~\eqref{eq:reducedsystem} \;
Set~$Z\in S_{n+1}$ as ~$z_{ij}=0$,~$\forall ij$, except for~$i,j=1,\dots,n+1-r$,
where~$z_{ij}=y_{ij}$ \;
Compute~$X=PZP^\top$\;
\Return{$X$}\;
\caption{ Compute approximate solution of~\eqref{P:SDP-matrix}
 from dual solution}
\label{algo:primalsol}
\end{algorithm}

Notice that we are not enforcing explicitly that~$Y\succeq 0$,
but if~$Y$ turns out to be positive semidefinite, then~$Z$ is positive
semidefinite and therefore~$X$ as well. We have the following theorem.

\begin{theorem}
Let~$y^*$ be a feasible solution of~\eqref{P:Dsdp} and
$X^*\in S_{n+1}$ the corresponding matrix produced by
Algorithm~\ref{algo:primalsol}.
If~$X^*\succeq 0$, then~$(X^*,y^*)$ are primal-dual optimal solutions
of Problems~\eqref{P:SDP-matrix}~and~\eqref{P:Dsdp}.
\end{theorem}
\begin{proof}
Let~$X^*$ be produced by Algorithm~\ref{algo:primalsol} such that
it is positive semidefinite.
We have that~$X^*$ is a feasible solution of Problem~\eqref{P:SDP-matrix},
since it satisfies the set of active constraints for the optimal dual
solution~$y^*$:
\begin{eqnarray*}
\pin{A_{0}}{X} & = & \pin{A_{0}}{PZP^\top}
=  \pin{P^\top A_{0}P}{Z}
= \pin{\overline{P^\top A_{0}P}}{Y}
= 1\\
\pin{A_{ij}}{X}
& = & \pin{A_{ij}}{PZP^\top}
= \pin{P^\top A_{ij}P}{Z}
=\pin{\overline{P^\top A_{ij}P}}{Y}
= \beta_{ij}
\end{eqnarray*}
for all $ij\in\mathcal{A}(y^*)$, this holds since $Y\in S_{n+1-r}$
is the solution of the system of equations~\eqref{eq:reducedsystem}.
It also satisfies complementarity slackness:
\[
(Q-\mathcal{A}^\top y^*)X^*= P\Dg{\lambda}P^\top PZP^\top=P\Dg{\lambda}ZP^\top=0,
\]
where the last equation holds since $Z$ is computed as in Step~3
of Algorithm~\ref{algo:primalsol}.
Namely, if~$\lambda_l =0$,  then the corresponding row~$l$ of~$\Dg{\lambda}Z$
is equal to zero. The other rows of~$\Dg{\lambda}Z$ are equal to zero from
the definition of~$Z$.
\qed\end{proof}

\begin{coro}
Let~$y^*$ be a feasible solution of the dual problem~\eqref{P:Dsdp}.
If the system
\begin{align*}
(Q-\mathcal{A}^\top y^*)X & =0 \\
\pin{A_0}{X} &= 1, \\
\pin{A_{ij}}{X} &= \beta_{ij} \quad \forall i,j\in \mathcal{A}(y^*)
\end{align*}
has a unique solution, then Algorithm~\ref{algo:primalsol} produces that solution.
\end{coro}

In summary, we have proposed a faster approach than solving a semidefinite
\mred{program}, but without any guarantee that the solution obtained will
satisfy the positive semidefiniteness constraint.
However there are theoretical reasons to argue that this approach will work in
practice.
In~\cite{Alizadeh1997}, it was proved  that dual non-degeneracy in semidefinite
programming implies the existence of a unique optimal primal solution; see~\cite{Alizadeh1997} for the definition of non-degeneracy.
Additionally, it was proved that dual non-degeneracy is a generic property.
Putting these two facts together, it  means that for randomly generated instances
the probability of obtaining a unique optimal primal solution is one.
From the practical point of view, we have implemented
Algorithm~\ref{algo:primalsol} and run experiments to check the positive
semidefiniteness of the computed matrix~$X$. We will see that for the random
instances considered in Section~\ref{Sec:experiments} this approach
works very well in practice.

\section{Adding linear constraints}
\label{Sec:lc}
Many optimization problems, such as the quadratic knapsack
problem~\cite{Pisinger,HelmbergRendlWeismantel},
can be modeled as a quadratic problem with linear constraints.
Linear constraints can be easily included into the current setting of our
problem.
Consider the following extension of Problem~\eqref{P:generalqp},
\begin{align}
\label{P:QuadraticLinearConst}
\min \ &\quad x^\top \hat Q x +\hat l^\top x+ \hat c\nonumber\\
\st &\quad a_j^\top x\leq b_j \quad \forall j=1,\dots,p \\
  & \quad \quad x\in D_1\times\dots\times D_n\;.\nonumber
\end{align}
Notice that the linear constraint~$a_j^\top x\leq b_j$ can be
equivalently written as
\[
\pin{A_j}{
\begin{pmatrix}
1\\
x
\end{pmatrix}
\begin{pmatrix}
1\\
x
\end{pmatrix}^\top
}
\leq
\beta_j,
\]
where
\[
A_j=
\begin{pmatrix}
\beta_j-b_j  & \frac{{a_j}_0}{2} & \dots & \frac{{a_j}_{n-1}}{2} \\
\frac{{a_j}_0}{2} & 0 & \dots & 0 \\
\vdots &  &  \ddots & \\
\frac{{a_j}_{n-1}}{2} & 0 & \dots & 0
\end{pmatrix}.
\]
Following a similar procedure as the one described in
Section~\ref{Sec:sdprelaxation}, we can formulate a semidefinite relaxation of
Problem~\eqref{P:QuadraticLinearConst} as follows
\begin{align}
\label{P:SDPLinearConst-matrix}
\min~~~\pin{Q}{X} & \nonumber \\
\st~~\pin{A_{0}}{X}&=1 \nonumber\\
\pin{A_{ij}}{X}&\leq \beta_{ij} \quad\forall j=l_i,\dots,u_i\quad\forall i=1,\dots,n  \\
\pin{A_{j}}{X}&\leq \beta_{j} \quad\forall j=1,\dots,p \nonumber \\
X&\succeq0. \nonumber
\end{align}
The matrices~$Q$,~$A_0$ and~$A_{ij}$ are defined as in
Section~\ref{Sec:matrixform}.
Observe that the new constraint matrices~$A_j$ have rank two.
The dual of Problem~\eqref{P:SDPLinearConst-matrix} can be calculated as
\begin{align}
\label{P:DsdpLinearConst}
\max~\pin{b}{y}~~~~\quad & \nonumber \\
\st~ \quad  Q - \mathcal{A}^\top y &\succeq 0 \\
y_{0} &\in \Rbb \nonumber\\
y_{ij}&\leq0 \quad \forall j=l_{i},\dots,u_i \quad\forall i=1,\dots, n\nonumber\\
y_{j}&\leq0 \quad \forall j=1,\dots,p, \nonumber
\end{align}
where $\mathcal{A}$ and $b$ are extended in the obvious way.
Again, we want to solve the log-det form of Problem~\eqref{P:DsdpLinearConst}
\begin{align}
\label{P:Dsdplc-barrier}
\max~~f(y;\sigma):=&\pin{b}{y}+\sigma \log\det(Q-\mathcal{A}^\top y) \nonumber\\
\st~~~~Q-\mathcal{A}^\top y &\succ 0 \\
 y_{0} &\in \Rbb  \nonumber \\
y_{ij}&\leq0 \quad \forall j=l_{i},\dots,u_i \quad\forall i=1,\dots, n\nonumber\\
y_{j}&\leq0 \quad \forall j=1,\dots, p.\nonumber
\end{align}
Notice that the overall form of the dual problem to be solved has not changed.
The new dual variables~$y_j$ corresponding to the additional linear constraints
play a similar role as the dual variables~$y_{ij}$, both must satisfy
the non-positivity constraint. Even more, the dual
problem~\eqref{P:DsdpLinearConst} remains strictly feasible,
this fact can be easily derived from Theorem~\ref{lem:initial}.

\begin{coro} \label{Cor:strictlyfeas}
  Problem~\eqref{P:DsdpLinearConst} is strictly feasible.
\end{coro}
If also the primal problem~\eqref{P:SDPLinearConst-matrix} is strictly
feasible, we can show as before that the level sets in our coordinate ascent
method are bounded and that we can always find a feasible step length.
However,
due to the addition of linear constraints, primal strict feasibility might
no longer be satisfied. However, by Corollary~\ref{Cor:strictlyfeas}
strong duality holds. In particular, we obtain
\begin{coro}\label{Cor:priminf}
If the primal problem~\eqref{P:SDPLinearConst-matrix} is infeasible,
then Problem~\eqref{P:Dsdplc-barrier} is unbounded.
\end{coro}
\begin{proof}
From Corollary~\ref{Cor:strictlyfeas}, it follows that both problems~\eqref{P:SDPLinearConst-matrix}
and~\eqref{P:DsdpLinearConst} have the same optimal
value; see e.g.~Theorem~2.2.5 in~\cite{Helmberg(2000)}. If~\eqref{P:SDPLinearConst-matrix} is infeasible, this value
is~$+\infty$, so that~\eqref{P:DsdpLinearConst} is unbounded.
Thus, by convexity, we can find an unbounded ray $y^{0}+sy$, $s\ge 0$,
for~\eqref{P:DsdpLinearConst}, starting at a
strictly feasible solution~$y^{0}$.
Now consider the concave function
$h(s)=\lambda_{\min}(Q-\mathcal{A}^{\top}(y^{0}+sy))$.
If there exists $s'>0$ such that $h(s')<h(0)$, then by concavity
$h(s)\rightarrow -\infty$ for $s\rightarrow \infty$ which is a contradiction
to the feasibility of the ray.
Thus $h(s)\ge h(0) = \lambda_{\min}(Q-\mathcal{A}^{\top}y^{0}) > 0$ for all
$s\ge 0$.
Hence, $\log\det(Q-\mathcal{A}^{\top}(y^{0}+sy))$ is bounded from below so
that the objective function of~\eqref{P:Dsdplc-barrier} goes to infinity.
\qed\end{proof}
The proof of Corollary~\ref{Cor:priminf} shows how to adapt the coordinate
search in this case:
either an appropriate root such as in Theorem~\ref{lem:onedimstep} or
Theorem~\ref{lem:steptwodim} exists, which can be used to determine the step length, or we have proven primal infeasibility.
\ifappendix
The details of the adapted algorithms are given in Appendices~\ref{sec:lincd}
and~\ref{sec:lincd2d} for the one- and two-dimensional approach, respectively.
\else
\mred{For the details of the adapted algorithms we refer to Chapter~6 of~\cite{Montenegro(2017)}.}
\fi

In case Problem~\eqref{P:SDPLinearConst-matrix} is feasible but not strictly
feasible, the barrier approach fails. In this case,
Problem~\eqref{P:Dsdplc-barrier} may be unbounded and hence the algorithm
wrongly concludes primal infeasibility.

\section{Experiments}
\label{Sec:experiments}



We now present the results of an experimental evaluation of our approach. Our experiments were carried out on Intel Xeon processors
running at 2.60~GHz.
For all the algorithms, the optimality tolerance OPTEPS was set to $10^{-6}$.
We have used as a base the code that already exists for Q-MIST.
Algorithms~\ref{algo:cd} and~CD2D were implemented in C++, using
routines from the LAPACK package~\cite{lapack} only in the initial
phase for computing a starting point, namely, to compute the smallest
eigenvalue of $\hat Q$ needed to determine $y^{(0)}$, and the inverse matrix
$W^{(0)}=(Q-\aop^\top y^{(0)})^{-1}$. The updates in each iteration can be
realized by elementary calculations, as explained in
Section~\ref{Sec:generalalgo}.

For our experiments, we have generated random instances
in the same way as proposed in~\cite{BuchheimWiegele(2013)}. We can
  control the percentage of negative eigenvalues in the objective matrix $\hat
  Q$, represented by the parameter $p$, so that $\hat Q$ is positive semidefinite for $p=0$, negative semidefinite for $p=100$ and indefinite for any other value~$p\in(0,100)$.

We will consider two types of variable domains: for \emph{ternary} instances, we have~$D_i=\{-1,0,1\}$, while for \emph{integer} instances we set~$D_i=\{-10,\dots,10\}$, for all~$i$.

In our implementation, we use the following rule to update the barrier
parameter:
whenever the entry of the gradient corresponding to the chosen coordinate has
an absolute value below~$0.1$ in the case of ternary instances
or below~$0.001$ for integer instances, we multiply~$\sigma$ by~$0.25$.
As soon as $\sigma$ falls below~$10^{-8}$, we fix it to this value.
The initial $\sigma$ is set to 1.

Recall that in Section~\ref{Sec:matrixform}, the parameter $\beta_{ij}$ can
be chosen arbitrarily. As it was pointed out, this parameter does not change the
feasible region of the primal problem~\eqref{P:SDP-matrix}, however it does
have an influence on its dual problem.
We have tested several choices of $\beta_{ij}$, such as setting it to zero for
all the constraints, or, according to~Lemma~\ref{lemma:rank}, so that all
constraint matrices have rank one.
We have found out experimentally that when choosing  the value of the parameter
$\beta_{ij}$ in such way that the constraint matrices $A_{ij}$ have
their first entry equal to zero, our approach has faster convergence.
Hence, we set  $\beta_{iu_i}= -l_iu_i$  for the upper bounding facets and
$\beta_{ij}=j(j+1)$ for lower bounding facets,
see Section~\ref{Sec:matrixform}.

\subsection{Stopping criterion}

It is important to find a good stopping criterion that either may
allow an early pruning of the nodes as soon as the current upper bound
is reached, or stops the algorithm when it cannot be expected any more
to reach this bound.  Our approach has the advantage of
producing feasible solutions of Problem~\eqref{P:Dsdp} and thus a
valid lower bound for Problem~\eqref{P:SDP-matrix} at every
iteration. This means that we can stop the iteration process and prune
the node as soon as the current lower bound exceeds a known upper
bound for Problem~\eqref{P:SDP-matrix}.

We propose the following stopping criterion. Every $n$ iterations,
we compare the gap at the current point (\emph{new-gap}) with
the previous one  $n$ iterations before (\emph{old-gap}).
If
$(1-\text{GAP})\text{\emph{old-gap}}<\text{\emph{new-gap}}$
and the number of iterations is at least $|D_i|\cdot n$, or
$
\text{\emph{new-gap}} < \text{OPTEPS}
$,
we stop the algorithm.
The gap is defined as the difference of the best upper bound known so far and
the current lower bound.
The value of GAP has to be taken in $[0,1]$.

In Figure~\ref{Fig:Ternaryn50allp} we illustrate the influence of
the parameter GAP on the running time and number of nodes needed in the entire
branch-and-bound tree, for both Algorithm~\ref{algo:cd}
and~CD2D.
We have chosen 110 random ternary instances of size 50, 10 instances for each
$p\in\{0,10,\dots,100\}$.
The horizontal axis corresponds to different values of GAP,
while the vertical axis corresponds to the average running time
(Figure~\ref{Fig:Ternaryn50allp}~(a)) and the average number of nodes
(Figure~\ref{Fig:Ternaryn50allp}~(b)), taken over the 110 instances.
If GAP$=$0, then the algorithm will stop only
when the new-gap reaches the absolute optimality tolerance.
As expected, strong bounds are obtained, and thus
the number of nodes is reduced and the time per node increases.
When GAP$=$1, the algorithm will stop immediately
after~$|D_i|n$ iterations, the lower bound produced
may be too weak and therefore the number of nodes is large.
A similar behavior of GAP is repeated for integer instances.
We conclude that choosing GAP=0.1 produces a good
balance between the quality of the lower bounds and the number of nodes.
We use the same stopping rules for both Algorithm~\ref{algo:cd}
and~CD2D.


\begin{figure}
\begin{center}
\subfigure[ ][Running time]{%
\includegraphics[scale=0.45]{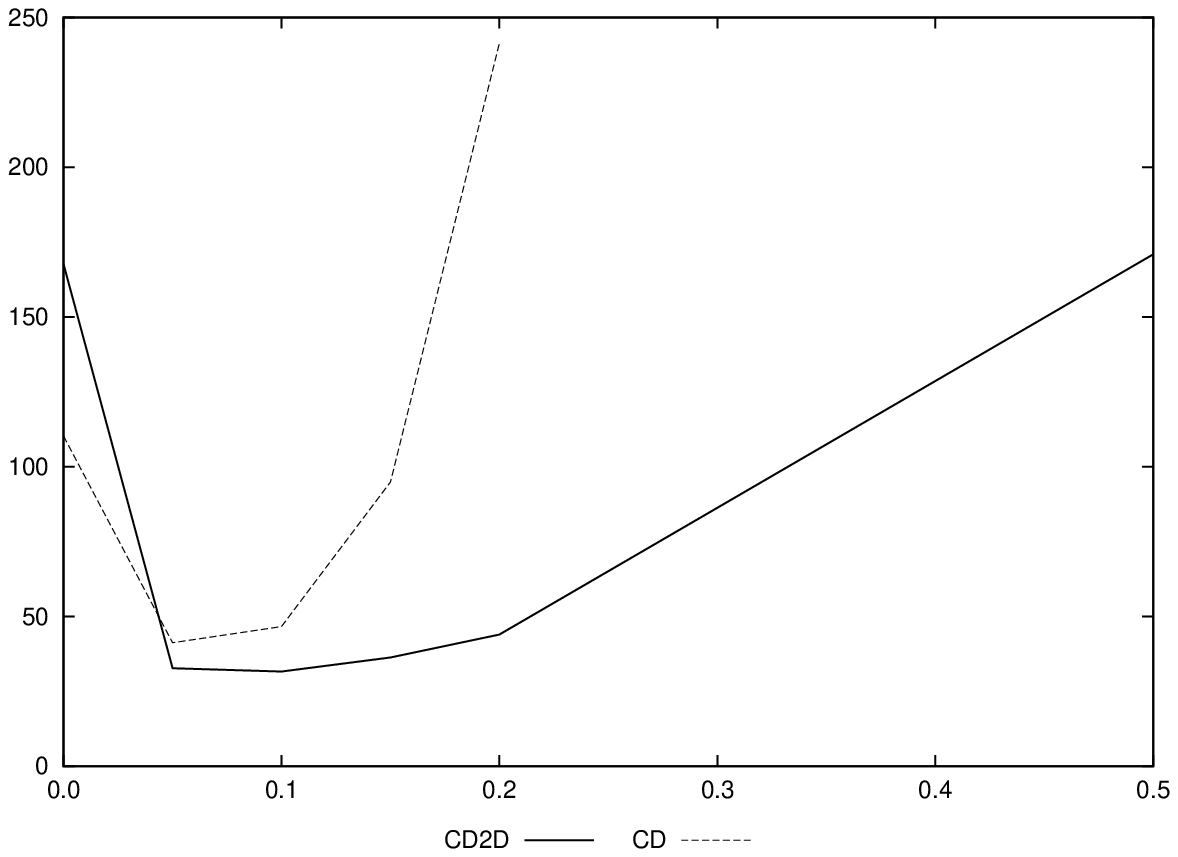}
}
\subfigure[ ][Number of nodes]{%
\includegraphics[scale=0.45]{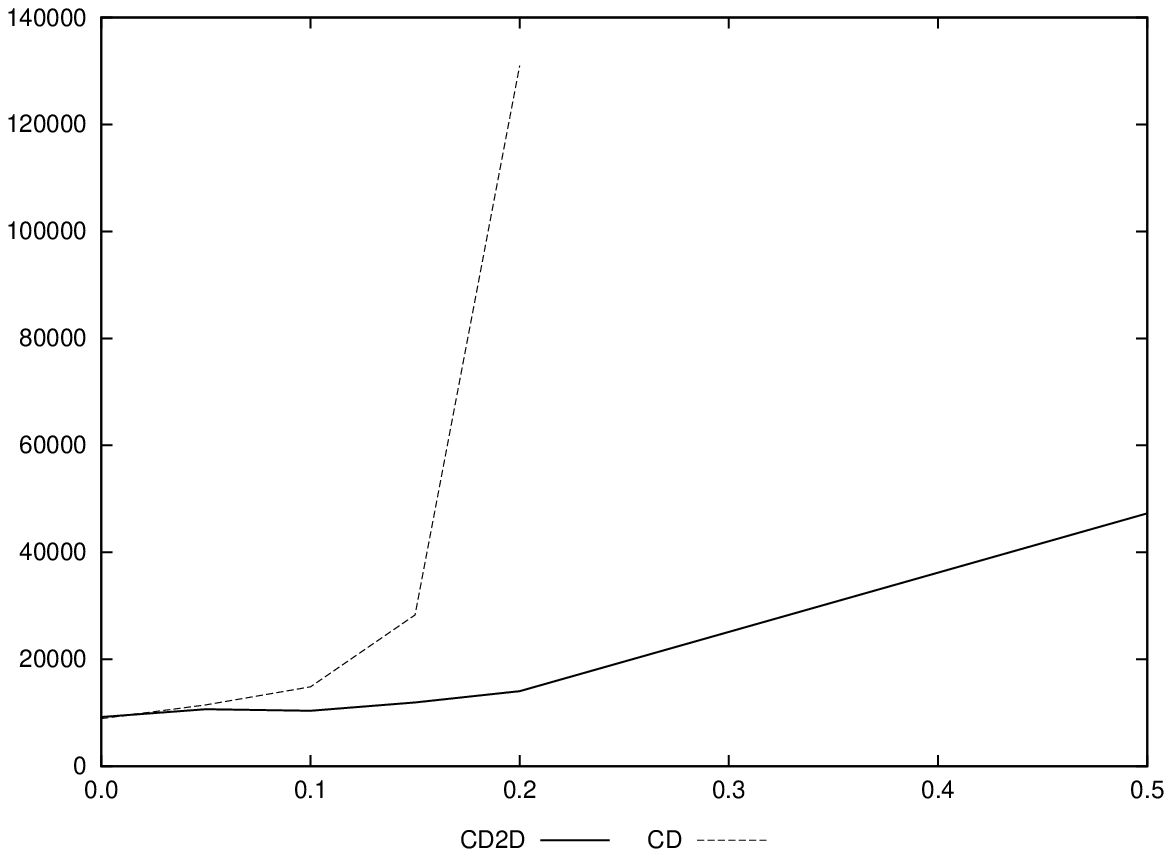}%
}
\end{center}
\vspace{-0.5cm}
\caption{\label{Fig:Ternaryn50allp} Influence of the gap criterion on the running time and the number of nodes for ternary instances, the behavior for integer instances is similar.}
\end{figure}

\subsection{Total running time}
Next, we are interested in evaluating the performance of the
branch-and-bound framework Q-MIST using the new Algorithms~\ref{algo:cd}
and~CD2D, and compare them to CSDP\mred{~\cite{csdp}, an implementation of an interior point method}.
Furthermore, we compare to other non-convex integer
programming software: COUENNE~\cite{couenne} and BARON~\mred{\cite{baron,baron:14.3.1}}.

In the following tables,  $n$  in the first column represents
the number of variables.
For each approach, we report the number of solved instances (\emph{\#}),
the average number of nodes explored in the branch-and-bound scheme
(\emph{nodes}) and the average running time in seconds (\emph{time}).
All lines report average results over 110 random instances.
We have set a time limit of one hour, and compute the averages considering
only the instances solved to proven optimality within this period of time.

In Table~\ref{tab:ternary} we present the results for ternary
instances.  As it can be observed, Q-MIST manages to solve all~110
instances for $n\leq 50$ with all three approaches. Both
Algorithms~\ref{algo:cd} and~CD2D require less time than CSDP even if
the number of nodes enumerated is much larger. For $n>50$, Q-MIST with
the new approach solves much more instances than with CSDP. Note that BARON
and COUENNE solved all 110 instances only for $n\leq20$ and $n\leq30$,
respectively.

Table~\ref{tab:integer} reports the results for integer instances,
the results show that Algorithm~CD2D outperforms all the other
approaches.  In this case, the lower bounds of Algorithm~\ref{algo:cd}
are too weak, leading to an excessive number of nodes, and it is not able
to solve all instances even of size 10 within the time limit.
On the contrary, Algorithm~CD2D manages to solve much more instances
than its competitors, also in the case of integer instances.

From the experiments reported in~\cite{BuchheimWiegele(2013)}, it was already
known that CSDP outperforms a previous version of COUENNE. The comparison of
Q-MIST with BARON is new. We have used also ANTIGONE~\cite{antigone}
for the comparison, but we do not report the results observed
since they are not better than those obtained with COUENNE.

\begin{sidewaystable}

  \centering

  \vspace{10cm}
  \caption{Results for ternary instances, $D_i=\{-1,0,1\}$}
  \label{tab:ternary}
\resizebox{\textwidth}{!}{%
  \begin{tabular}{ | c |r r r | r r r |r r r || r r r | r r r|}
    \hline
 &   \multicolumn{9}{c||}{Q-MIST}  &  \multicolumn{3}{c|}{COUENNE}  & \multicolumn{3}{c|}{BARON}   \\
$n$ & \multicolumn{3}{c}{CD} &  \multicolumn{3}{c}{CD2D}  & \multicolumn{3}{c||}{CSDP} &  \multicolumn{3}{c|}{} &  \multicolumn{3}{c|}{} \\
\cline{2-16}
 & \#  & nodes & time & \#  & nodes & time & \#  & nodes & time  & \#  & nodes & time  & \#  & nodes & time \\
\hline
\hline
10  & 110 &  49.31 & 0.03   & 110 &  28.05 & 0.02   & 110 &  10.11 & 0.07   & 110 &  11.91 & 0.10   & 110 &  1.42 & 0.07   \\
20  & 110 &  250.31 & 0.16   & 110 &  174.24 & 0.06   & 110 &  67.95 & 0.32   & 110 &  2522.35 & 10.40   & 110 &  8.87 & 0.80   \\
30  & 110 &  1531.29 & 1.25   & 110 &  668.47 & 0.65   & 110 &  247.24 & 2.17   & 85 &  150894.54 & 1225.72   & 110 &  8.67 & 27.59   \\
40  & 110 &  3024.42 & 4.98   & 110 &  2342.75 & 3.47   & 110 &  1030.25 & 12.20   & 4 &  134864.75 & 2330.83   & 65 &  45.88 & 280.17   \\
50  & 110 &  14847.49 & 46.61   & 110 &  10357.11 & 31.62   & 110 &  7284.09 & 136.81   & 0 &   --  &  --    & 21 &  29.14 & 222.93   \\
60  & 107 &  34353.45 & 197.60   & 110 &  33780.15 & 155.84   & 109 &  17210.14 & 526.96   & 0 &   --  &  --    & 12 &  10.67 & 219.77   \\
70  & 83 &  76774.30 & 515.98   & 98 &  94294.82 & 656.58   & 71 &  17754.41 & 887.17   & 0 &   --  &  --    & 3 &  2.33 & 257.51   \\
80  & 63 &  98962.24 & 1151.22   & 65 &  126549.25 & 1150.02   & 34 &  19553.47 & 1542.38   & 0 &   --  &  --    & 0 &   --  &  --    \\
\hline
\end{tabular}}

\vspace{1cm}
  \centering
  \caption{Results for integer instances, $D_i=\{-10,\dots,10\}$}
   \label{tab:integer}
\resizebox{\textwidth}{!}{%
  \begin{tabular}{ | c |r r r | r r r |r r r || r r r | r r r|}
    \hline
  &   \multicolumn{9}{c||}{Q-MIST}  &  \multicolumn{3}{c|}{COUENNE}  & \multicolumn{3}{c|}{BARON}   \\
$n$ & \multicolumn{3}{c}{CD} &  \multicolumn{3}{c}{CD2D}  & \multicolumn{3}{c||}{CSDP} &  \multicolumn{3}{c|}{} &  \multicolumn{3}{c|}{} \\
\cline{2-16}
 & \#  & nodes & time & \#  & nodes & time & \#  & nodes & time  & \#  & nodes & time  & \#  & nodes & time \\
\hline
\hline
10  & 107 &  1085009.52 & 105.54  & 110 &  70.58 & 0.07   & 109 &  26.29 & 0.16   & 110 &  5817.25 & 7.51   & 110 &  45.43 & 0.49   \\
20  & 10 &  296203.60 & 154.30  & 110 &  969.11 & 0.99   & 110 &  324.71 & 2.85   & 98 &  91473.86 & 489.05    & 109 &  140.43 & 6.44   \\
30  & 4 &  179909.00 & 336.25  & 110 &  5653.71 & 13.89   & 110 &  2196.87 & 34.49   & 0 &   --  &  --    & 104 &  137.47 & 38.20   \\
40  & 0 &   --  &  --    & 110 &  38458.96 & 187.76   & 108 &  13029.41 & 386.68   & 0 &   --  &  --    & 59 &  202.93 & 255.65   \\
50  & 0 &   --  &  --    & 96 &  99205.07 & 944.79   & 67 &  24292.79 & 1247.10   & 0 &   --  &  --    & 15 &  17.87 & 279.82   \\
60  & 0 &   --  &  --    & 53 &  84802.25 & 1329.92   & 26 &  30105.15 & 2088.00   & 0 &   --  &  --   & 8 &  11.25 & 282.82   \\
70  & 0 &   --  &  --    & 2 &  48648.00 & 1218.50   & 1 &  2011.00 & 254.00   & 0 &   --  &  --    & 7 &  12.43 & 457.47   \\
\hline
\end{tabular}}
\end{sidewaystable}

As a summary, we can state that Algorithm~CD2D yields a significant
improvement of the algorithm Q-MIST when compared with CSDP, and it is
even capable to compete with other commercial and free software as
BARON and COUENNE.  However, it is important to point out that the
performance of BARON is almost not changed when considering ternary or
integer variable domains, it solves more or less the same number of
instances in both cases. On the contrary, it is obvious that the change of
the domains affected the performance of our approach significantly,
especially in Algorithm~\ref{algo:cd}.

\mred{
  To conclude the first part of our experiments, we have generated two
  other types of instances using the same generator as before and
  changing only the objective matrix~$Q$. Firstly, we have produced
  random \emph{sparse} matrices as follows: each entry of the
  matrix~$Q$ is zero with probability~$1-\tfrac p{100}$ and the remaining entries are chosen randomly from the interval $[-1,1]$. To obtain symmetric matrices, we set~$Q=\frac{1}{2}(Q+Q^{\top})$. We generated 10 instances for each~$p\in\{25,50,75,100\}$. We report the results of the experiments for sparse ternary instances in Table~\ref{tab:sparseternary} and for sparse integer instances in Table~\ref{tab:sparseinteger}.}

\mred{Additionally, we produced \emph{low rank} matrices~$Q$ by setting 50\% of the eigenvalues to zero, then we chose the remaining eigenvalues to be negative with probability~$\tfrac p{100}$, for~$p\in\{0,10,\dots,100\}$. For each value of~$p$ we have generated 10 instances, thus for each size~$n$ we report average results for 110 instances again.
The results of these experiments are reported in Tables~\ref{tab:lowrankternary}~and~\ref{tab:lowrankinteger}.
}
\mred{
\begin{table}[h!]
  \centering
  \caption{Sparse ternary instances}
  \label{tab:sparseternary}
  \begin{tabular}{ | c | c |r r r | r r r ||r r r |}
    \hline
&  &   \multicolumn{6}{c||}{Q-MIST}  & \multicolumn{3}{c|}{BARON}   \\
$p$ & $n$  &  \multicolumn{3}{c}{CD2D}  &  \multicolumn{3}{c||}{CSDP} &  \multicolumn{3}{c|}{} \\
\cline{3-11}
& & \#  & nodes & time  & \#  & nodes & time & \#  & nodes & time  \\
\hline
\hline
25   & 10  & 10 &  24.40 & 0.00   & 10 &  10.80 & 0.10   & 10 &  1.00 & 0.10   \\
 & 20  & 10 &  163.80 & 0.10   & 10 &  68.00 & 0.40   & 10 &  1.00 & 0.14   \\
 & 30  & 10 &  767.40 & 0.60   & 10 &  505.40 & 3.40   & 10 &  1.40 & 0.44   \\
 & 40  & 10 &  3137.40 & 5.00   & 10 &  1444.20 & 14.30   & 10 &  6.20 & 4.87   \\
 & 50  & 10 &  17734.00 & 55.80   & 10 &  10200.20 & 166.10   & 10 &  16.40 & 30.81   \\
 & 60  & 10 &  88798.40 & 481.30   & 9 &  46796.33 & 1175.44   & 7 &  512.00 & 499.15   \\
 & 70  & 4 &  160533.00 & 1249.75   & 3 &  73215.67 & 2840.67   & 0 &   --  &  --    \\
\hline
50   & 10  & 10 &  32.00 & 0.00   & 10 &  14.00 & 0.00   & 10 &  1.00 & 0.12   \\
 & 20  & 10 &  203.60 & 0.10   & 10 &  98.80 & 0.60   & 10 &  1.00 & 0.20   \\
 & 30  & 10 &  1243.80 & 1.10   & 10 &  461.80 & 2.50   & 10 &  2.20 & 2.99   \\
 & 40  & 10 &  3657.00 & 6.00   & 10 &  2192.80 & 21.20   & 10 &  8.40 & 34.22   \\
 & 50  & 10 &  32299.00 & 103.40   & 10 &  10850.40 & 175.10   & 2 &  7.00 & 96.78   \\
 & 60  & 10 &  107354.60 & 552.80   & 10 &  66379.60 & 1676.20   & 0 &   --  &  --    \\
 & 70  & 4 &  387745.00 & 2988.00   & 0 &   --  &  --    & 0 &   --  &  --    \\
 & 80  & 1 &  212447.00 & 2182.00   & 0 &   --  &  --    & 0 &   --  &  --    \\
\hline
75   & 10  & 10 &  21.40 & 0.00   & 10 &  7.60 & 0.00   & 10 &  1.10 & 0.11   \\
 & 20  & 10 &  308.40 & 0.00   & 10 &  133.60 & 0.90   & 10 &  1.40 & 0.60   \\
 & 30  & 10 &  962.80 & 0.80   & 10 &  439.00 & 2.60   & 10 &  2.40 & 5.46   \\
 & 40  & 10 &  5947.20 & 10.50   & 10 &  2176.80 & 21.30   & 6 &  39.67 & 337.80   \\
 & 50  & 10 &  40459.80 & 128.60   & 10 &  17928.20 & 284.40   & 0 &   --  &  --    \\
 & 60  & 10 &  116544.80 & 581.80   & 10 &  50177.40 & 1265.70   & 0 &   --  &  --    \\
 & 70  & 4 &  260129.50 & 1899.50   & 1 &  79939.00 & 3043.00   & 0 &   --  &  --    \\
 & 80  & 1 &  104621.00 & 1098.00   & 0 &   --  &  --    & 0 &   --  &  --    \\
\hline
100   & 10  & 10 &  36.40 & 0.00   & 10 &  11.60 & 0.00   & 10 &  1.00 & 0.11   \\
 & 20  & 10 &  208.00 & 0.10   & 10 &  106.00 & 0.60   & 10 &  1.20 & 0.53   \\
 & 30  & 10 &  1235.20 & 0.70   & 10 &  495.20 & 2.80   & 10 &  2.20 & 6.94   \\
 & 40  & 10 &  4492.00 & 7.90   & 10 &  1909.20 & 18.70   & 6 &  14.00 & 191.62   \\
 & 50  & 10 &  39410.00 & 118.20   & 10 &  13536.40 & 215.70   & 0 &   --  &  --    \\
 & 60  & 10 &  129061.80 & 619.70   & 10 &  51079.40 & 1303.60   & 0 &   --  &  --    \\
 & 70  & 5 &  268774.20 & 1888.20   & 1 &  89807.00 & 3183.00   & 0 &   --  &  --    \\
\hline
\end{tabular}
\end{table}
}

%

\mred{
\begin{table}[h!]
  \centering
  \caption{Sparse integer instances}
  \label{tab:sparseinteger}
  \begin{tabular}{ | c | c |r r r | r r r ||r r r |}
    \hline
&  &   \multicolumn{6}{c||}{Q-MIST}  & \multicolumn{3}{c|}{BARON}   \\
$p$ & $n$  &  \multicolumn{3}{c}{CD2D}  &  \multicolumn{3}{c||}{CSDP} &  \multicolumn{3}{c|}{} \\
\cline{3-11}
& & \#  & nodes & time  & \#  & nodes & time & \#  & nodes & time  \\
\hline
\hline
25   & 10  & 10 &  60.00 & 0.00   & 10 &  21.60 & 0.00   & 10 &  1.20 & 0.05   \\
 & 20  & 10 &  860.00 & 1.10   & 10 &  312.20 & 1.60   & 10 &  2.40 & 0.12   \\
 & 30  & 10 &  6462.60 & 15.20   & 10 &  1923.00 & 29.70   & 10 &  1.60 & 0.42   \\
 & 40  & 9 &  20913.89 & 109.67   & 10 &  8300.00 & 239.90   & 10 &  63.30 & 14.97   \\
 & 50  & 8 &  122938.75 & 1401.38   & 6 &  30037.33 & 1514.17   & 9 &  221.44 & 122.36   \\
 & 60  & 2 &  85003.00 & 1282.00   & 1 &  31911.00 & 2509.00   & 2 &  10.00 & 28.45   \\
\hline
50   & 10  & 10 &  109.00 & 0.00   & 10 &  27.20 & 0.00   & 10 &  1.20 & 0.06   \\
 & 20  & 10 &  831.00 & 0.90   & 10 &  247.20 & 1.70   & 10 &  1.20 & 0.18   \\
 & 30  & 10 &  5928.20 & 14.30   & 10 &  2252.80 & 36.00   & 10 &  68.30 & 18.70   \\
 & 40  & 10 &  19523.00 & 111.30   & 10 &  11753.60 & 364.20   & 9 &  202.89 & 175.07   \\
 & 50  & 8 &  127993.25 & 1495.25   & 6 &  24956.00 & 1315.50   & 4 &  110.00 & 349.70   \\
\hline
75   & 10  & 10 &  90.00 & 0.20   & 10 &  35.80 & 0.00   & 10 &  1.60 & 0.07   \\
 & 20  & 10 &  1382.00 & 1.40   & 10 &  371.80 & 2.30   & 10 &  1.20 & 0.26   \\
 & 30  & 10 &  6679.00 & 16.50   & 10 &  1828.20 & 29.00   & 10 &  57.70 & 32.40   \\
 & 40  & 10 &  38621.60 & 227.50   & 10 &  13384.40 & 403.30   & 5 &  31.00 & 127.03   \\
 & 50  & 6 &  88795.67 & 1165.00   & 5 &  38027.80 & 2018.40   & 0 &   --  &  --    \\
\hline
100   & 10  & 10 &  89.40 & 0.00   & 10 &  28.40 & 0.00   & 10 &  1.00 & 0.05   \\
 & 20  & 10 &  1547.00 & 1.70   & 10 &  361.40 & 2.20   & 10 &  1.20 & 0.31   \\
 & 30  & 10 &  6418.20 & 14.90   & 10 &  2842.60 & 47.30   & 10 &  3.40 & 10.93   \\
 & 40  & 10 &  23796.20 & 142.40   & 10 &  9067.20 & 271.50   & 3 &  15.00 & 119.51   \\
 & 50  & 6 &  128231.50 & 1708.50   & 7 &  37107.57 & 1870.86   & 0 &   --  &  --    \\
 & 60  & 2 &  174494.00 & 3208.00   & 0 &   --  &  --    & 0 &   --  &  --    \\
\hline
\end{tabular}
\end{table}
}
%

\mred{
\begin{table}[h!]
  \centering
  \caption{Low rank ternary instances}
\label{tab:lowrankternary}
  \begin{tabular}{ | c |r r r | r r r ||r r r |}
    \hline
  &   \multicolumn{6}{c||}{Q-MIST}  & \multicolumn{3}{c|}{BARON}   \\
$n$  &  \multicolumn{3}{c}{CD2D}  &  \multicolumn{3}{c||}{CSDP} &  \multicolumn{3}{c|}{} \\
\cline{2-10}
 & \#  & nodes & time  & \#  & nodes & time & \#  & nodes & time  \\
\hline
\hline
10  & 110 &  18.82 & 0.01   & 110 &  8.29 & 0.00   & 110 &  1.00 & 0.05   \\
20  & 110 &  109.98 & 0.10   & 110 &  35.85 & 0.06   & 110 &  1.17 & 0.52   \\
30  & 110 &  538.04 & 0.39   & 110 &  241.24 & 1.67   & 108 &  1.17 & 4.50   \\
40  & 110 &  1858.00 & 3.12   & 110 &  1206.69 & 13.48   & 107 &  1.35 & 39.58   \\
50  & 110 &  6976.22 & 21.12   & 110 &  4284.35 & 77.80   & 89 &  4.39 & 125.09   \\
60  & 110 &  17459.54 & 88.43   & 110 &  15426.85 & 450.68   & 44 &  10.05 & 295.82   \\
70  & 106 &  54070.34 & 444.75   & 90 &  23441.13 & 1051.92   & 13 &  7.00 & 172.25   \\
80  & 64 &  108409.72 & 1302.81   & 29 &  9737.90 & 791.69   & 10 &  1.40 & 67.10   \\
\hline
\end{tabular}
\end{table}
\begin{table}[h!]
  \centering
  \caption{Low rank integer instances}
  \label{tab:lowrankinteger}
  \begin{tabular}{ | c |r r r | r r r ||r r r |}
    \hline
  &   \multicolumn{6}{c||}{Q-MIST}  & \multicolumn{3}{c|}{BARON}   \\
$n$  &  \multicolumn{3}{c}{CD2D}  &  \multicolumn{3}{c||}{CSDP} &  \multicolumn{3}{c|}{} \\
\cline{2-10}
 & \#  & nodes & time  & \#  & nodes & time & \#  & nodes & time  \\
\hline
\hline
10  & 110 &  79.38 & 0.11   & 110 &  17.42 & 0.00   & 110 &  1.55 & 0.08   \\
20  & 110 &  2987.02 & 3.53   & 110 &  181.71 & 1.24   & 66 &  259.98 & 25.57   \\
30  & 106 &  45392.58 & 115.92   & 110 &  1336.84 & 19.40   & 93 &  11.62 & 12.52   \\
40  & 99 &  21928.90 & 104.56   & 109 &  9588.69 & 256.38   & 98 &  4.94 & 36.45   \\
50  & 96 &  60249.35 & 561.40   & 100 &  26046.96 & 1171.06   & 72 &  18.94 & 189.98   \\
60  & 61 &  96534.56 & 1483.84   & 29 &  20258.24 & 1392.90   & 12 &  39.00 & 632.07   \\
70  & 12 &  60000.25 & 1157.42   & 6 &  4934.67 & 544.50   & 0 &   --  &  --    \\
\hline
\end{tabular}
\end{table}
}
It turns out that sparsity does not seem to have an important
  impact on the hardness of the problems when solved with our
  coordinate ascent approach. The size of problems we can solve to
  optimality is very similar for all densities considered, both in the
  ternary and in the integer case. On the other hand, BARON can
  slightly profit from sparser instances. However, our new approach
  can solve significantly more instances than BARON for each value
  of~$p$, except for $p=25$ in the integer case.

Concerning low-rank instances, the effect is not clear: in the ternary case, more instances can be solved by our approach for $n=70$, but for~$n=80$ one instance less is solved within the time limit. In the integer case, our approach produces slightly weaker results for low-rank instances. BARON clearly profits from low-rank input matrices. In summary, both sparse and low rank matrices do not change the running times of our approach significantly, while BARON can (slightly) profit from both properties.

\subsection{Primal solution}

At the root node,  we have performed the evaluation of
Algorithm~\ref{algo:primalsol}, designed to compute an approximate primal
solution of Problem~\eqref{P:SDP-matrix} using
the dual feasible solution $y^*$ of Problem~\eqref{P:Dsdp};
see the details in Section~\ref{Sec:primal}.
Recall that we need to compute the eigenvalue decomposition of
the matrix~$Q-\mathcal{A}^\top y^*$, and set a tolerance to
decide which other eigenvalues will be considered as zero.
In the experiments we have taken into account that~$Q-\mathcal{A}^\top y^*$~has
always at least one zero eigenvalue,
and considered as zero all the eigenvalues smaller or equal to 0.01.
We have run experiments to check the positive semidefiniteness of the
matrix~$X^*$~at the root node of the branch-and-bound tree, with
the dual variables obtained from Algorithms~\ref{algo:cd} and~CD2D.
We did this test for all instances used in the experiments of the previous section.
We have observed that in all the cases the smallest eigenvalue of $X$~is
always greater than $-10^{-14}$. Based on this fact we can conclude that
the method works.

\subsection{Behavior with linear constraints}

In Section~\ref{Sec:lc} we have described how our approach can
be extended  when  inequality constraints are added to
Problem~\eqref{P:Quadratic}. For the experiments in this section
we will consider ternary instances with 
two types of constraints: inequalities of the form $\sum_{i=1}^nx_i\leq0$
and knapsack constraints $a^\top x\leq b$.
The vector $a\in\Rbb^n$ and the right hand side of $a^\top x\leq b$ are
generated as follows: each entry $a_i$ is chosen randomly distributed
in~$\{1,2,\dots,5\}$ and $b$ is randomly distributed in
$\{1,\dots,\sum_{i=1}^na_i\}$.
The objective function is generated as explained before.
Tables~\ref{tab:ternarylc0} and~\ref{tab:ternarylc1} report
the results of the performance  of Algorithm Q-MIST
with~CD2D and CSDP, and compare with BARON.
The dimension $n$ of the problem is chosen  from~10
to 50 and $p\in\{0,10,\dots,100\}$; as before each line in the
tables corresponds to the average computed over 110 instances solved
within the time limit, 10
instances for each combination of $n$ and $p$.

Comparing the results reported in Table~\ref{tab:ternary} with
those of Tables~\ref{tab:ternarylc0} and~\ref{tab:ternarylc1},
one can conclude that the addition of a linear constraint
does not change the overall behavior of our approach.
As it can be seen, Q-MIST -- with both approaches~CD2D and CSDP --
outperforms BARON.
However, Algorithm~CD2D, as shown in Table~\ref{tab:ternary}, is much faster
even if the number of nodes explored is larger.

\begin{table}[htb]
  \centering
  \caption{Results for ternary instances plus $\sum_i^n x_i\leq0$}
 \label{tab:ternarylc0}
  \begin{tabular}{ | c |r r r | r r r ||r r r |}
    \hline
  &   \multicolumn{6}{c||}{Q-MIST}  & \multicolumn{3}{c|}{BARON}   \\
$n$  &  \multicolumn{3}{c}{CD2D}  &  \multicolumn{3}{c||}{CSDP} &  \multicolumn{3}{c|}{} \\
\cline{2-10}
 & \#  & nodes & time  & \#  & nodes & time & \#  & nodes & time  \\
\hline
\hline
10  & 110 &  35.85 & 0.01   & 110 &  12.73 & 0.02   & 110 &  1.29 & 0.09   \\
20  & 110 &  195.56 & 0.35   & 110 &  74.18 & 0.34   & 110 &  6.70 & 1.10   \\
30  & 110 &  993.21 & 1.08   & 110 &  332.38 & 2.65   & 110 &  17.31 & 43.86   \\
40  & 110 &  3160.16 & 4.85   & 110 &  1199.55 & 16.47   & 48 &  13.44 & 233.40   \\
50  & 110 &  13916.13 & 40.35   & 110 &  7235.00 & 159.66   & 20 &  61.20 & 174.96   \\
\hline
\end{tabular}
\end{table}
%
\begin{table}[htb]
  \centering
  \caption{Results for ternary instances plus knapsack constraint}
 \label{tab:ternarylc1}
  \begin{tabular}{ | c |r r r | r r r ||r r r |}
    \hline
  &   \multicolumn{6}{c||}{Q-MIST}  & \multicolumn{3}{c|}{BARON}   \\
$n$  &  \multicolumn{3}{c}{CD2D}  &  \multicolumn{3}{c||}{CSDP} &  \multicolumn{3}{c|}{} \\
\cline{2-10}
 & \#  & nodes & time  & \#  & nodes & time & \#  & nodes & time  \\
\hline
\hline
10  & 110 &  29.36 & 0.01   & 110 &  11.15 & 0.05   & 110 &  1.41 & 0.08   \\
20  & 110 &  185.78 & 0.24   & 110 &  70.75 & 0.29   & 110 &  9.15 & 1.04   \\
30  & 110 &  685.64 & 0.74   & 110 &  247.80 & 2.16   & 110 &  16.04 & 38.17   \\
40  & 110 &  2361.33 & 3.85   & 110 &  1035.29 & 14.95   & 56 &  37.23 & 289.56   \\
50  & 110 &  9844.31 & 31.10   & 110 &  7140.91 & 165.15   & 21 &  67.48 & 191.01   \\
\hline
\end{tabular}
\end{table}

\section{Conclusion}
We have developed an algorithm that on the one hand exploits
the structure of the semidefinite relaxations proposed by Buchheim and Wiegele,
namely a small total number of active constraints and constraint
matrices characterized by a low rank.
On the other hand, our algorithm exploits this special structure by solving
the dual problem of the semidefinite relaxation, using a barrier method in
combination with a coordinate-wise exact line search, motivated
by the algorithm presented by Dong. The main ingredient of our
algorithm is the computationally cheap update at each iteration and
an easy computation of the exact step size. Compared to interior
point methods, our approach is much faster in obtaining strong dual
bounds. Moreover, no explicit separation and re-optimization is
necessary even if the set of primal constraints is large, since in
our dual approach this is covered by implicitly considering all
primal constraints when selecting the next coordinate.
Even more, the structure of the problem allows us to perform a
plane search instead of a single line search, this speeds up
the convergence of the algorithm.
Finally, linear constraints are easily integrated into the
algorithmic framework.

We have performed experimental comparisons on randomly
generated instances, showing that our approach  significantly
improves the performance of Q-MIST when compared with CSDP
and outperforms other specialized global optimization software, such as BARON.



\bibliographystyle{spmpsci}      
\bibliography{qmist-cd}   

\ifappendix
\newpage
\appendix
\section{Step size for CD}
\label{Sec:stepsizeCD}
Each constraint matrix~$A_{ij}$ can be factored as follows:
\[
A_{ij}=E_{ij}IC_{ij},
\]
where~$E_{ij}\in \Rbb^{(n+1)\times2}$ is defined by~$E_{ij}:=(e_{0}\; e_{i})$,~$e_{0},
e_{i}\in \Rbb^{n+1}$, ~$C_{ij}\in \Rbb^{2\times(n+1)}$ is defined
by~$C_{ij}:=(A_{ij})_{\{0,i\},\{0,\dots,n\}}$, and~$I$ is the~$2\times2$-identity matrix, i.e.,
\[
E_{ij}:=
\begin{pmatrix}
1 & 0\\
\vdots & \vdots\\
0 & 1\\
\vdots & \vdots\\
0 & 0
\end{pmatrix}
\quad
\text{ and }
\quad
C_{ij}=
\begin{pmatrix}
(A_{ij})_{00} & 0 & \hdots &0 & (A_{ij})_{0i} &0& \hdots 0\\
(A_{ij})_{0i} & 0 & \hdots &0 & (A_{ij})_{ii} &0& \hdots 0
\end{pmatrix}
\]
By the Woodbury formula~\cite{Hager(1989)}
\begin{equation}
\label{Eq:inverseranktwo}
(W^{-1} -sA_{ij})^{-1}=(W^{-1} -sE_{ij}IC_{ij})^{-1}=W+WE_{ij}(\tfrac{1}{s}I-C_{ij}WE_{ij})^{-1}C_{ij}W\;.
\end{equation}
Notice that the matrix~$\tfrac{1}{s}I-C_{ij}WE_{ij}$ is a~$2\times2$-matrix,
so its inverse can be easily computed even as a closed formula.

On the other hand, from Lemma~\ref{lemma:rank}, we know under which
conditions a constraint matrix~$A_{ij}$ has rank-one. In that case,
we obtain the following factorization:
\begin{equation}
\label{eq:factormatA}
A_{ij} = (A_{ij})_{ii}vv^\top,
\end{equation}
where~$v:= (A_{ij})_{0i}e_0 + (A_{ij})_{ii}e_i$.
The inverse of~$(W^{-1} -sA_{ij})$ is then computed using the Woodbury formula
for rank-one update,
\begin{equation}
\label{Eq:inverserankone}
(W^{-1} -sA_{ij})^{-1} = (W^{-1}-s(A_{ij})_{ii}vv^\top)^{-1}
= W+\frac{(A_{ij})_{ii}s}{1-(A_{ij})_{ii}sv^\top Wv} Wvv^\top W.
\end{equation}


Now, we need to find the value of~$s$ that makes the gradient in
\eqref{Eq:gradient} zero, this requires to solve the following equation
\begin{equation*}
\beta_{ij}-\sigma\pin{A_{ij}}{(W^{-1} -sA_{ij})^{-1}}=0.
\end{equation*}
In order to solve this equation, we distinguish two possible cases,
depending on the rank of the constraint matrix of the chosen coordinate.
We use the factorizations of the matrix~$A_{ij}$ explained above.

\paragraph{Rank-two.}
By replacing the inverse matrix~\eqref{Eq:inverseranktwo} in
the gradient~\eqref{Eq:gradient} and setting it to zero, we obtain
\begin{equation}
\label{Eq:gradeqzero}
\beta_{ij}-\sigma\pin{A_{ij}}{W}-\sigma\pin{A_{ij}}{WE_{ij}(\tfrac{1}{s}I+C_{ij}WE_{ij})^{-1}C_{ij}W}=0.
\end{equation}
Due to the sparsity of the constraint matrices~$A_{ij}$, the inner matrix product
is simplified a lot, in fact we have to compute only the entries~$00$,~$0i$,
$0i$ and~$ii$ of the  matrix product
$WE_{ij}(\tfrac{1}{s}I+C_{ij}WE_{ij})^{-1}C_{ij}W$.
We only need to compute rows $0$ and $i$  of the matrix product $WE_{ij}$ and columns $0$ and $i$ of $C_{ij}W$,
\[
WE_{ij}=
\begin{pmatrix}
w_{00} & w_{0i}\\
\vdots & \vdots \\
w_{0i} & w_{ii}\\
\vdots & \vdots
\end{pmatrix}
\quad
C_{ij}W=
\begin{pmatrix}
(A_{ij})_{00}w_{00} +(A_{ij})_{0i} w_{0i} &\hdots & (A_{ij})_{00}w_{0i} +(A_{ij})_{0i} w_{ii} & \dots\\
(A_{ij})_{0i}w_{00} +(A_{ij})_{ii} w_{0i} &\hdots & (A_{ij})_{0i}w_{0i} +(A_{ij})_{ii} w_{ii} & \dots
\end{pmatrix}.
\]
From the last matrix we have
\[
C_{ij}WE_{ij}=
\begin{pmatrix}
(A_{ij})_{00}w_{00} +(A_{ij})_{0i} w_{0i} & (A_{ij})_{00}w_{0i} +(A_{ij})_{0i} w_{ii}\\
(A_{ij})_{0i}w_{00} +(A_{ij})_{ii} w_{0i} & (A_{ij})_{0i}w_{0i} +(A_{ij})_{ii} w_{ii}
\end{pmatrix}.
\]
Moreover, the inverse of the matrix $(\tfrac{1}{s}I+C_{ij}WE_{ij})$ is computed easily, its entries are rational expressions on~$s$. Finally, from~\eqref{Eq:gradeqzero} we obtain a rational equation on~$s$ of degree two, namely
\[
\frac{\beta_{ij} \alpha_1 ws^2+(2\sigma \alpha_1 w -\alpha_2\beta_{ij})s+\beta_{ij}-\sigma\alpha_2}{\alpha_1 ws^2-\alpha_2s+1}=0,
\]
where
\begin{align*}	
\alpha_1 & := (A_{ij})_{00}(A_{ij})_{ii}- (A_{ij})_{0i}^2,\\
\alpha_2 & := (A_{ij})_{00}w_{00}+2(A_{ij})_{0i}w_{0i}+(A_{ij})_{ii}w_{ii},\\
w & := w_{00}w_{ii}-w_{0i}^2.
\end{align*}
Theorem~\ref{lem:onedimstep} shows that, since~$s\mapsto f(y +se_{ij};\sigma)$
is continuously differentiable on the level sets, the denominator of the latter
equation can not become zero before finding a point where the gradient is zero.
Therefore, the step size~$s$ is obtained setting the numerator to zero, and
using the quadratic formula for the roots of the general quadratic equation:
\[
s=\frac{-2\sigma \alpha_1 w +\alpha_2\beta_{ij}
\pm \sqrt{(2\sigma \alpha_1 w -\alpha_2\beta_{ij})^2 -4\beta_{ij} \alpha_1 w(\beta_{ij}-\sigma\alpha_2)}}{2\beta_{ij} \alpha_1 w}.
\]
Then, according to Theorem~\ref{lem:onedimstep} we will need to take
the smallest/biggest $s$ on the right direction of the chosen coordinate.
\paragraph{Rank-one.} In case the rank of~$A_{ij}$ is one, the computations can
 be simplified. We proceed as before, replacing \eqref{Eq:inverserankone} in
the gradient~\eqref{Eq:gradient} and setting it to zero:
\[
\beta_{ij}-\sigma\pin{(A_{ij})_{ii}vv^\top}{W+\frac{(A_{ij})_{ii}s}{1-(A_{ij})_{ii}s v^\top Wv} Wvv^\top W} = 0.
\]
Denote~$t:=\pin{vv^\top}{W}= v^\top W v = v_0^2w_{00}+2v_0v_iw_{0i}+v_i^2w_{ii}$,
then~$\pin{vv^\top}{ Wvv^\top W}= (v^\top W v)^2= t^2$. Replacing this in the last
equation yields
\begin{equation}
\label{eq:rankonegrad}
\beta_{ij}-\sigma (A_{ij})_{ii} t -\sigma t^2\frac{(A_{ij})_{ii}^2s}{1-(A_{ij})_{ii}ts}  = 0.
\end{equation}
The last expression turns out to be a rational equation linear in~$s$,
and the step size is
\[
s=\frac{1}{(A_{ij})_{ii}t}-\frac{\sigma}{\beta_{ij}}.
\]
Notice that~$s\neq \frac{1}{(A_{ij})_{ii}t}$ and hence the denominator
in~\eqref{eq:rankonegrad} is different from zero.
We have to point out that the zero coordinate can also be chosen as ascent
direction, in that case the gradient is
\[
\nabla_{s}f(y +se_0;\sigma)=1-\sigma\pin{A_0}{(W^{-1} -sA_0)^{-1}}.
\]
As before, the inverse of~$W^{-1} -sA_0$ is computed using the  Woodbury
formula for rank-one update
\begin{align*}
(W^{-1} -sA_{0})^{-1} =(W^{-1} - se_{0}e_{0}^\top)^{-1}
= W+\frac{s}{1-sw_{00}} (We_{0}) (We_{0})^\top.
\end{align*}
The computation of the step size becomes simpler, we just need to
find a solution of the linear equation
\[
1-\sigma\pin{A_0}{(W^{-1} -sA_0)^{-1}}=0.
\]
Solving the last equation, the step size is
\[
s=\frac{1}{w_{00}}-\sigma.
\]
A similar formula for the step size is obtained for other cases when
the constraint matrix~$A_{ij}$ has rank-one and corresponds to
an upper facet such that~$l_i=-u_i$. Since in this
case~$(A_{ij})_{00} = (A_{ij})_{0i}=0$  and~$(A_{ij})_{ii} =1$,
the factorization of~$A_{ij}$ in~\eqref{eq:factormatA} reduces to
\[
A_{ij}=e_ie_i^\top,
\]
and~$t=w_{ii}$. Thus,  the step is:
\[
s=\frac{1}{w_{ii}}-\frac{\sigma}{\beta_{ij}}.
\]

With the step size~$s^{(k)}$ determined, we use the following
formulae for a fast update, again making use of the Woodbury formula:
\begin{align*}
y^{(k+1)}&:=y^{(k)}+s^{(k)}e_{ij^{(k)}}\\
W^{(k+1)}&:=W^{(k)}+W^{(k)}E_{ij^{(k)}}\big(\tfrac{1}{s^{(k)}}I-C_{ij^{(k)}}W^{(k)}E_{ij^{(k)}}\big)^{-1}C_{ij^{(k)}}W^{(k)},
\end{align*}
or
\[
W^{(k+1)} := W^{(k)}+\frac{(A_{ij})_{ii}s^{(k)}}{1-(A_{ij})_{ii}s^{(k)}} (W^{(k)}v^{(k)})(W^{(k)}v^{(k)})^\top.
\]

\section{Two dimensional approach}
\label{Sec:stepsizeCD2D}
For computing $s_0(s)$, we need to compute $w(s)_{00}$. We have that
\begin{align*}
w(s)_{00} &= (W^{-1} -sA_{ij})^{-1}_{00}\\
&= w_{00}+(WE_{ij}\big(\tfrac{1}{s}I+C_{ij}WE_{ij}\big)^{-1}C_{ij}W)_{00}.
\end{align*}
As explained in the previous section, the computations are simplified due to the structure of the matrices involved.
We obtain that
\[
w(s)_{00} = -\frac{(A_{ij})_{ii}ws-w_{00}}{\alpha_1 ws^2-\alpha_2s+1}-\sigma,
\]
with $\alpha_1$, $\alpha_2$ and $w$ defined as in the last section. Thus
\[
s_0(s)=\frac{1}{w(s)_{00}}=-\frac{\alpha_1 ws^2-\alpha_2s+1}{(A_{ij})_{ii}ws-w_{00}}-\sigma.
\]
In order to choose the coordinate direction $e_{ij}$, we need to compute $g_{ij}'(0)$, we have
\begin{align*}
g_{ij}'(0)&=s_0'(0)+\beta_{ij} -\sigma \pin{s_0'(0)A_0+A_{ij}}{(W^{-1} -s_0(0)A_0-0A_{ij})^{-1}}\\
&=s_0'(0)+\beta_{ij} -\sigma \pin{s_0'(0)A_0+A_{ij}}{W+\frac{s_0(0)}{1-s_0(0)w_{00}}(We_0)^\top(We_0)^\top},
\end{align*}
and in the last inner matrix product we only need to consider the entries $00$, $0i$ and $ii$, thus
\begin{align*}
g_{ij}'(0) &=s_0'(0)+\beta_{ij} -\sigma \left((s_0'(0)+(A_{ij})_{00})(w_{00}+\frac{1-\sigma w_{00}}{\sigma w_{00}^2}w_{00}^2)\right.\\
&\left.~+2(A_{ij})_{00}(w_{0i}+\frac{1-\sigma w_{00}}{\sigma w_{00}^2}w_{00}w_{0i}) +(A_{ij})_{ii}(w_{ii}+\frac{1-\sigma w_{00}}{\sigma w_{00}^2}w_{ii}^2)\right)\\
&=\beta_{ij}-(A_{ij})_{00}-2(A_{ij})_{0i}\frac{w_{0i}}{w_{00}}-(A_{ij})_{ii}w_{ii}\sigma +(A_{ij})_{ii}(1-\sigma w_{00})\frac{w_{ii}^2}{w_{00}^2}.
\end{align*}
More explicitly, for upper and lower bound facets, we get 
\[
g_{ij}'(0)=
\begin{cases}
j(j+1)-2\frac{w_{0i}}{w_{00}}j-\frac{w_{0i}}{w_{00}}-(\sigma w_{00}-1)\frac{w_{ii}^2}{w_{00}^2}+\sigma w_{ii} &  j=l_i,\dots,u_i-1,\\
l_{i}u_{i}+ \frac{w_{0i}}{w_{00}}(l_{i}+u_{i}) +(\sigma w_{00}-1)\frac{w_{ii}^2}{w_{00}^2}-\sigma w_{ii} &  j=u_i.
\end{cases}
\]

\section{Algorithm CD including linear constraints}
\label{sec:lincd}
The addition of~$p$ linear constraints in the primal problem implies that
for the search of a coordinate direction there are~$p$ additional
potential directions.
As before, the entries of the gradient for the new coordinates can be
explicitly computed as
\begin{align*}
\nabla_{y}f(y;\sigma)_{j}&=\beta_{j}-\sigma \pin{W}{A_{j}}\\
&=\beta_{j}-\sigma((A_j)_{00}w_{00}+2\sum_{k=1}^{n}(A_j)_{0k}w_{0k}).
\end{align*}
We then choose the coordinate of the gradient with largest absolute value,
considering coordinates both corresponding to the lower bounding facets,
the upper bounding facet and the new linear constraints.
In Section~\ref{Sec:steponedimension}, we observed that at most
$1+4n$ candidates have to be considered to select the coordinate direction.
Thus, in this case, we will have at most~$1+4n+p$ candidates.


The computation of the step size follows an  analogous procedure as in
Section~\ref{Sec:steponedimension}.
Therefore, if one of
the new possible candidates for  coordinate direction~$e_j\in\Rbb^{m+p+1}$
for~$j\in\{1,\dots,p\}$ has been chosen, we need to compute~$s$ such
that either
\[
\nabla_{s}f(y+se_j;\sigma)=0 \quad\text{ and } \quad s\leq -y_{j}
\]
or
\[
\nabla_{s}f(y+se_j;\sigma) >0 \quad \text{ and } \quad s= -y_{j}.
\]
We have that
\begin{equation}
\label{Eq:gradlc}
\nabla_{s}f(y+se_j;\sigma)_{j}=\beta_{j}-\sigma\pin{A_{j}}{(W^{-1} -sA_{j})^{-1}}.
\end{equation}




The existence of an optimal step size now depends on primal feasibility.
There is no guarantee that the level sets of the function are bounded, or as
we already mentioned, if the primal problem is not feasible, the dual
problem will be unbounded.
Testing primal feasibility is a difficult task, however,
from~Lemma~\ref{lem:existence} we know that if there exists $s$ in the
correct direction of the line search that makes the gradient~\eqref{Eq:gradlc}
zero, then there exists also one on the feasible region.
This implies the following result.

\begin{theorem}\label{Th:steplc}
~
\begin{itemize}
\item[(i)] Let the coordinate~$j$ be such that~$\nabla_{y}f(y;\sigma)_{j}>0$
  and~$y_{j}<0$.
  If the gradient~\eqref{Eq:gradlc} has a positive root,
  then for the smallest positive root~$s^+$, either~$y+s^+e_{j}$~is dual
  feasible and~$\nabla_{s}f(y +s^+e_{j};\sigma)=0$, or~$y_{j}+s^+>0$,
  $y -y_{j}e_{j}$ is dual feasible, and~$\nabla_{s}f(y -y_{j}e_{j};\sigma)>0$.
  Otherwise, $y+se_{ij}$ is dual feasible with
  $\nabla_{s}f(y~+~se_{j};~\sigma)~>~0~$ for all $s\in[0,-y_{ij}]$.
\item[(ii)] Let the coordinate~$j$ be such that~$\nabla_{y}f(y;\sigma)_{j}<0$.
  If the gradient~\eqref{Eq:gradlc} has a negative root,
  then for the biggest negative root~$s^-$, the point~$y+s^-e_{j}$~is
  dual feasible and~$\nabla_{s}f(y +s^-e_{j};\sigma)=0$.
  Otherwise, $y+se_{ij}$ is dual feasible with $\nabla_{s}f(y+se_{j};\sigma)>0$ for all $s\le 0$.
\end{itemize}
\end{theorem}

As before, in order to find the step size, it is necessary to compute the
inverse of~$W^{-1} -sA_{j}$.
As it was mentioned, the  constraint matrices~$A_j$ are rank-two matrices.
They admit the following factorization
\[
A_j= E_j I C_j,
\]
where
\[
E_j =
\begin{pmatrix}
\tfrac{1}{2}(A_j)_{00}  & 1\\
(A_j)_{01} & 0\\
\vdots & \vdots\\
(A_j)_{0n} & 0
\end{pmatrix}
\quad \text{ and } \quad
C_j =
\begin{pmatrix}
 1 & 0 & \hdots & 0\\
\tfrac{1}{2}(A_j)_{00} &  (A_j)_{01} & \hdots & (A_j)_{0n}
\end{pmatrix}
.
\]
With the Woodbury formula and the factorization above, we have that
the inner product of~$A_j$ and~$(W^{-1} -sA_{j})^{-1}$ reduces to the
inner product of two~$2\times2$ matrices:
\begin{align*}
\pin{A_{j}}{(W^{-1} -sA_{j})^{-1}}
&= \pin{E_jIC_j}{W + WE_j(\tfrac{1}{s}I-C_jWE_j)^{-1} C_jW} \\
&= \pin{I}{E_j^\top WC_j^\top + E_j^\top WE_j(\tfrac{1}{s}I-C_jWE_j)^{-1}C_jWC_j^\top}.
\end{align*}
We obtain
\begin{align*}
E_j^\top WE_j &=
\begin{pmatrix}
d & f\\
f & w_{00}
\end{pmatrix},\quad
C_jWC_j^\top =
\begin{pmatrix}
w_{00} & f\\
f & d
\end{pmatrix},\\
C_jWE_j  &=
\begin{pmatrix}
f & w_{00}\\
d & f
\end{pmatrix},\quad
E_j^\top WC_j^\top  =
\begin{pmatrix}
f & d\\
w_{00} & f
\end{pmatrix},
\end{align*}
where
\begin{align*}
d &= \tfrac{1}{4}w_{00}(A_j)_{00}^2 + (A_j)_{00}\sum_{i=1}^n w_{0i}(A_j)_{0i}
+\sum_{i=1}^n\sum_{k=1}^n w_{ik}(A_j)_{0i}(A_j)_{0k} \\
&= \pin{W}{(A_j)_{0\cdot}(A_j)_{0\cdot}^\top}, \\
f &= \tfrac{1}{2}w_{00}(A_j)_{00} +\sum_{i=1}^n w_{0i}(A_j)_{0i} \\
&= W_{0\cdot}^\top (A_j)_{0\cdot}.
\end{align*}
Replacing the inner product in the gradient~\eqref{Eq:gradlc}, we obtain a
rational function of degree two
\[
\nabla_{s}f(y+se_j;\sigma)_{j}
=\frac{\beta_j (d w_{00}-f^2)s^2+(2d\sigma w_{00}- 2f^2\sigma +2\beta_j f)s+ 2f\sigma-\beta_j}{(dw_{00}-f^2)s^2+2fs-1}.
\]
Finally the step size is obtained setting the numerator to zero, yielding
\[
s = \frac{-d\sigma w_{00}+f^2 \sigma-\beta_j f\pm\sqrt{d^2 \sigma^2 w_{00}^2-2 d f^2 \sigma^2 w_{00}+f^4 \sigma^2+\beta_j^2 d w_{00}}}{\beta_j (d w_{00}-f^2)}.
\]

In the implementation of the algorithm, if no root of the
gradient~\eqref{Eq:gradlc}
is found in the right direction, the step size has to be set to
$-y_{j}$ when the coordinate~$j$ is such that~$\nabla_{y}f(y;\sigma)_{j}>0$
and~$y_{j}<0$, or $s=M$, where $M\ll0$,
when the coordinate~$j$ is such that~$\nabla_{y}f(y;\sigma)_{j}<0$.

It is clear that Algorithm~\ref{algo:cd} can be easily extended to compute
lower bounds for the optimal value of Problem~\eqref{P:SDPLinearConst-matrix}.

\section{Algorithm CD2D including linear constraints}\label{sec:lincd2d}
A two-dimensional update is also possible for solving the dual of
Problem~\eqref{P:SDPLinearConst-matrix}, again in this case, any linear
combination of a constraint matrix~$A_j$ with~$A_0$ remains being a rank-two
matrix.
The optimal two-dimensional step size~$(s_0(s),s)$ along the coordinate plane
spanned by~$(e_0,e_{j})$ can be computed following an analogous procedure to the
 one explained in Section~\ref{Sec:SimultaneusUpdate}. It turns out,
in this case, that the computation of the step size is technically less
complicated.
Lemma~\ref{lemma:steps02D} can be used to compute the step size~$s_0(s)$ along
the direction $e_0$, in terms of a given step size $s$ along coordinate
direction~$e_j$,
namely,
\[
s_0(s) =\frac{1}{{w(s)}_{00}}-\sigma.
\]
Recall that~$W(s)=(W^{-1}-sA_j)^{-1}$, with need to compute the first entry of this matrix. We have
\begin{align*}
W(s)_{00}&=(W^{-1}-sA_j)^{-1}_{00} \\
&= (W+ WE_j(\tfrac{1}{s}I-C_jWE_j)^{-1} C_jW)_{00}\\
&= w_{00}+ (WE_j(\tfrac{1}{s}I-C_jWE_j)^{-1} C_jW)_{00}.
\end{align*}
The inverse of $WE_j(\tfrac{1}{s}I-C_jWE_j)$ is easy to compute, since it is a $2\times2$-matrix.
\[
 (WE_j(\tfrac{1}{s}I-C_jWE_j))^{-1} = \tfrac{1}{(\tfrac{1}{s}-f)^2-dw_{00}}
\begin{pmatrix}
\tfrac{1}{s}-f & w_{00}\\
d & \tfrac{1}{s}-f
\end{pmatrix}
\]
with~$f$,~$d$ defined as in the last section. From the matrix products $WE_j$ and $C_jW$ we need to compute only the first row and column, respectively:
\[
WE_j =
\begin{pmatrix}
f & w_{00}\\
\vdots & \vdots
\end{pmatrix}
\quad
C_jW =
\begin{pmatrix}
 w_{00} & \hdots\\
f & \hdots
\end{pmatrix}.
\]
We obtain that
\begin{align*}
w(s)_{00} = w_{00}+ \tfrac{1}{(\tfrac{1}{s}-f)^2-dw_{00}}(2fw_{00}(\tfrac{1}{s}-f)+dw_{00}^2+f^2w_{00})\\
=\frac{w_{00}}{(dw_{00}-f^2)s^2+2fs+\sigma w_{00}-1}.
\end{align*}
And thus
\[
s_0(s) =
-\frac{1}{w_{00}}((dw_{00}-f^2)s^2+2fs+\sigma w_{00}-1).
\]
We then can define the function
\[
g_{j}(s) := f(y +s_0(s)e_0+se_{j};\sigma)
\]
over the set~$\{s\in\Rbb\mid Q-\mathcal{A}^\top (y+s_0(s)e_0+se_{j}) %
\succ0\}$. We have to solve a similar problem to \eqref{P:simultLinesearch},
namely, we need to find~$s\in\Rbb$ such that
\[
(g_{j}'(s)=0
\text{ and }
s\leq -y_{j})
\quad
\text{ or }
\quad
(g_{j}'(s) > 0
\text{ and }
s= -y_{j}).
\]
We thus need to compute the derivative of~$g_j(s)$
\begin{equation}
\label{Eq:gradlc2d}
g_{j}'(s)
=s_0'(s)+\beta_{j} -\sigma \pin{s_0'(s)A_0+A_{j}}{(W^{-1} -s_0(s)A_0-sA_{j})^{-1}}.
\end{equation}

As we already pointed out, the existence of a step size is related
with primal feasibility.
We have the following theorem that, analogous to Theorem~\ref{Th:steplc},
is a direct consequence of Lemma~\ref{lem:existence}.

\begin{theorem}\label{Th:steplc2d}
~
\begin{itemize}
\item[(i)] Let the coordinate~$j$ be such that~$g_{j}'(0)>0$ and~$y_{j}<0$.
  If the derivative~\eqref{Eq:gradlc2d} has a positive root,
  then for the smallest positive root~$s^+$, either $y+s_0(s^+)e_0+s^+e_{j}$
  is dual feasible and~$g_{j}'(s^+)=0$, or $y_{j}+s^+>0$,
  $y+s_0(-y_{j})e_0-y_{j}e_{j}$ is dual feasible   and~$g_{j}'(-y_j)>0$.
  Otherwise, $y+s_0(s)e_0+se_{j}$ is dual feasible with $g_j'(s)>0$ for all
  $s\in [0,-y_{ij}]$.
\item[(ii)] Let the coordinate~$j$ be such that~$g_{j}'(0)<0$.
  If the derivative~\eqref{Eq:gradlc2d} has a negative root,
  then for the biggest negative~$s^-$, the point~$y+s_0(s^-)e_0+s^-e_{j}$~is
  dual feasible and~$g_{j}'(s^-)=0$.
  Otherwise, $y+s_0(s)e_0+se_{j}$ is dual feasible with with~$g_{j}'(s)>0$
  for all~$s\leq 0$.
\end{itemize}
\end{theorem}

In order to compute the inner product in~\eqref{Eq:gradlc2d},
we propose the following factorizations for the
matrices~$\bar A_j := s_0'(s)A_0+A_{j}$ and $\tilde A_j := s_0(s)A_0+sA_{j}$:
\[
\bar A_j = \bar E_j I \bar C_j, \text{ and }
\tilde A_j = \tilde E_j I \tilde C_j,
\]
where
\[
\bar E_j =
\begin{pmatrix}
\tfrac{1}{2}(s_0'(s)+(A_j)_{00})  & 1\\
(A_j)_{01} & 0\\
\vdots & \vdots\\
(A_j)_{0n} & 0
\end{pmatrix},
\quad
\bar C_j =
\begin{pmatrix}
 1 & 0 & \hdots & 0\\
\tfrac{1}{2}(s_0'(s)+(A_j)_{00}) &  (A_j)_{01} & \hdots & (A_j)_{0n}
\end{pmatrix}
,
\]

\[
\tilde E_j =
\begin{pmatrix}
\tfrac{1}{2}(s_0(s)+s(A_j)_{00})  & 1\\
s(A_j)_{01} & 0\\
\vdots & \vdots\\
s(A_j)_{0n} & 0
\end{pmatrix},
\quad
\tilde C_j =
\begin{pmatrix}
 1 & 0 & \hdots & 0\\
\tfrac{1}{2}(s_0(s)+s(A_j)_{00}) &  s(A_j)_{01} & \hdots & s(A_j)_{0n}
\end{pmatrix}
.
\]
In this way, the inner product of matrices in~\eqref{Eq:gradlc2d}
can be rewritten as the inner product of two~$2\times 2$ matrices:
\begin{align*}
\pin{\bar A_j}{(W^{-1} -\tilde A_{j})^{-1}} &=
\pin{\bar E_jI \bar C_j}{W + W\tilde E_j(I-\tilde C_jW \tilde E_j)\tilde C_jW} \\
&= \pin{I}{\bar E_j^\top W \bar C_j^\top + \bar E_j^\top W\tilde E_j(I-\tilde C_jW \tilde E_j)\tilde C_jW \bar C_j^\top},
\end{align*}
where
\begin{align*}
\bar E_j^\top W \tilde E_j &=
\begin{pmatrix}
d_1 & \bar f\\
\tilde f & w_{00}
\end{pmatrix},\quad
\tilde C_jW \bar C_j^\top =
\begin{pmatrix}
w_{00} & \bar f\\
\tilde f & d_1
\end{pmatrix},\\
\tilde C_jW\tilde E_j  &=
\begin{pmatrix}
\tilde f & w_{00}\\
\tilde d & \tilde f
\end{pmatrix},\quad
\bar E_j^\top W \bar C_j^\top =
\begin{pmatrix}
\bar f & \bar d\\
w_{00} & \bar f
\end{pmatrix},
\end{align*}
and
\begin{align*}
\bar d &=  \pin{W}{(\bar A_j)_{0\cdot}(\bar A_j)_{0\cdot}^\top}, \\
\tilde d &= \pin{W}{(\tilde A_j)_{0\cdot}(\tilde A_j)_{0\cdot}^\top}, \\
\bar f &= W_{0\cdot}^\top (\bar A_j)_{0\cdot},\\
\tilde f &= W_{0\cdot}^\top (\tilde A_j)_{0\cdot},\\
d_1 &= \pin{W}{(\tilde A_j)_{0\cdot}(\bar A_j)_{0\cdot}^\top}.
\end{align*}
By doing all calculations, one can verify that
$\pin{A_{j}}{(W^{-1} -sA_{j})^{-1}}$ is actually zero.
Replacing this into~\eqref{Eq:gradlc2d} we get ~$g'_j(s) = s_0'(s)+\beta_{j}$,
where
\[
s_0'(s) = -\frac{2}{w_{00}}((dw_{00}-f^2)s+f),
\]
and setting~$g'_j(s)$ to zero, we obtain a linear equation on the step
size~$s$, whose root is
\begin{equation}
\label{eq:stepcd2dlc}
s=\frac{2f -\beta_jw_{00}}{2(f^2-dw_{00})}.
\end{equation}
Observe that the step size $s$ is independent on the value of~$\sigma$,
however the step $s_0$ is still dependent.
From Theorem~\ref{Th:steplc2d} it follows that:
\begin{itemize}
\item[(i)] if the coordinate $j$ is such that $g'_j(0)>0$ and $y_j<0$, and
  if the derivative~\eqref{Eq:gradlc2d} has a positive root, then
  the step size~\eqref{eq:stepcd2dlc} must be positive.
  When there is no positive root $s$ can be set to $-y_{ij}$.
\item[(ii)] if the coordinate $j$ is such that $g'_j(0)<0$,  and
  if the derivative~\eqref{Eq:gradlc2d} has a negative root, then
  the step size~\eqref{eq:stepcd2dlc} must be negative.
  When there is no negative root set $s=M$, with $M\ll0$.
\end{itemize}

The coordinate selection will be done in a similar way as in
Section~\ref{Sec:SimultaneusUpdate}, i.e., we will choose the coordinate
with the largest absolute value of~$g'_j(0)$. Recall that from
Section~\ref{Sec:SimultaneusUpdate}, we have $4n$ potential coordinates,
after adding $p$ linear constraints we will have that $4n+p$
candidates to be considered.
\fi

\end{document}